\documentclass[a4paper,10pt,reqno]{amsart}

\usepackage{amsfonts,amssymb,amsthm}
\usepackage[mathscr]{eucal}
\usepackage{graphicx}
\usepackage{epsfig}
\usepackage{psfrag}
\usepackage{amsfonts}
\usepackage{amsmath}
\usepackage{amsthm}
\usepackage{latexsym}
\usepackage{verbatim}
\usepackage{mathtools}

\DeclarePairedDelimiter\floor{\lfloor}{\rfloor} 
  \usepackage[usenames,dvipsnames]{xcolor}

 \usepackage{hyperref}
\usepackage[active]{srcltx}

\input xy
\xyoption{all}

\theoremstyle{plain}
\newtheorem{thm}{Theorem}
\newtheorem{theorem}[thm]{Theorem}
\newtheorem{corollary}[thm]{Corollary}
\newtheorem{lemma}[thm]{Lemma}
\newtheorem{prop}[thm]{Proposition}
\newtheorem{defin}[thm]{Definition}

\newtheoremstyle{exm}
{9pt}{9pt}{}{}{\bfseries}{}{.5em}{}
\theoremstyle{exm}
\newtheorem{exm}[thm]{Example}
\newtheoremstyle{rmk}
{9pt}{9pt}{}{}{\bfseries}{}{.5em}{}
\theoremstyle{rmk}
\newtheorem{rmk}[thm]{Remark}
\theoremstyle{alg}

\newtheoremstyle{question}
{9pt}{9pt}{}{}{\bfseries}{}{.5em}{}
\theoremstyle{question}
\newtheorem{question}[thm]{Question}
\numberwithin{equation}{section}
\numberwithin{thm}{section}
\numberwithin{figure}{section}

\newcommand{\cc}{\mathsf{c}}
\newcommand{\ce}{\mathsf{c}_{1}^{S^1}}

\newcommand{\LL}{\mathbb{L}}
\newcommand{\Lat}{\mathcal{L}}
\newcommand{\J}{\mathsf{J}}
\newcommand{\e}[1]{\mathsf{e}^{2 \pi \mathsf{i} {#1}}}

\newcommand{\co}{\mathcal{H}}
\newcommand{\coe}{\mathcal{H}_{S^1}}
\newcommand{\K}{K}
\newcommand{\Ke}{K_{S^1}}

\newcommand{\T}{\mathbb{T}}

\newcommand{\M}{\mathsf{M}}
\newcommand{\R}{\mathbb{R}}
\newcommand{\Z}{{\mathbb{Z}}}
\newcommand{\C}{{\mathbb{C}}}

\newcommand{\Q}{\mathbb{Q}}

\newcommand{\cp}{\cc_1^+}
\newcommand{\cm}{\cc_1^-}
\newcommand{\ac}{(\M,\J,S^1)}
\DeclareMathOperator{\pic}{Pic}
\DeclareMathOperator{\td}{Todd}
\DeclareMathOperator{\tor}{Tor}
\DeclareMathOperator{\ttot}{\mathcal{T}}

\DeclareMathOperator{\ch}{Ch}

\DeclareMathOperator{\ind}{Ind}
\DeclareMathOperator{\k0}{\mathbf{k_0}}
\DeclareMathOperator{\Hi}{H}
\newcommand{\Hil}{\Hi}
\DeclareMathOperator{\U}{U}
\DeclareMathOperator{\Gen}{P}

\newcommand{\purge}[1]{}

\title[On the Chern numbers and the Hilbert polynomial]{On the Chern numbers and the Hilbert polynomial of an almost complex manifold with a circle action}

\author[S. Sabatini]{Silvia Sabatini}
\address{Mathematisches Institut, Universit\"at zu K\"oln, Weyertal 86-90, D-50931 K\"oln, Germany}
\email{sabatini@math.uni-koeln.de}


\date{\today}

\begin{document}

\begin{abstract}
Let $(\M,\J)$ be a compact, connected, almost complex manifold of dimension $2n$ endowed with a $\J$-preserving circle action with isolated fixed points.
In this note we analyse the `geography problem' for such manifolds, deriving equations relating the Chern numbers to the index $\k0$ of $(\M,\J)$. We study the symmetries and zeros of the Hilbert polynomial associated to $(\M,\J)$, which imply many rigidity results for the Chern numbers when $\k0\neq 1$. 

We apply these results to the category of compact, connected symplectic manifolds.
A long-standing question posed by McDuff and Salamon \cite{MS}, also known as the `McDuff conjecture', asked about the existence of non-Hamiltonian actions with isolated fixed points.
This question was answered recently by Tolman \cite{T3}, with an explicit construction of a six-dimensional manifold with such an action. 
One issue that this raises is whether one can find topological criteria that ensure the manifold can only support a Hamiltonian or only a non-Hamiltonian action.
 In this vein, we are able to deduce such criteria from our rigidity theorems in terms of relatively few Chern numbers, depending
on the index. This  
improves upon results of Feldman \cite{Fe} for which one needs to know the entire Todd genus. 

Another consequence of our work is that, in the situation above with a Hamiltonian action, the minimal Chern number coincides with the index and is at most $n+1$, mirroring results of Michelsohn \cite{Mi} in the complex category and Hattori \cite{Ha}
in the almost complex category.   
\end{abstract}
\maketitle

\tableofcontents

\section{Introduction}\label{intro}

Let $(\M,\J)$ be a compact, connected, \emph{almost complex} manifold of dimension $2n$ and let $\cc=\sum_{i=1}^n \cc_i$ be the total Chern class of the tangent bundle of $\M$. 
To each partition of $n$ one can associate an integer, called \emph{Chern number}, given by
$\cc_{i_1}\cdots \cc_{i_k}[\M]:=\langle \cc_{i_1}\cdots \cc_{i_k}, \mu_\M\rangle$, where $i_1+\cdots +i_k=n$ and $\mu_\M$ is the orientation homology class of $\M$ (the orientation being induced by the almost complex structure). 
The problem of determining which lists of integers can arise 
as the Chern numbers of a compact almost complex manifold $(\M,\J)$ of a given dimension (also known as the \emph{geography} problem) has been investigated in different settings. Without additional assumptions on $(\M,\J)$, a theorem of Milnor \cite{H} implies that it is necessary and sufficient for these integers to satisfy a certain set of congruences depending on $n$ (the same is true if $\M$ is connected and $n\geq 2$, see \cite{Ge}). 
However, if the manifold is endowed with a $\J$-preserving circle action, further restrictions arise and the geography problem is, in its generality, still open. When the fixed point set is empty (or more generally when all the stabilisers are discrete), as a consequence of the Atiyah--Bott--Berline--Vergne localization formula (hereinafter the ABBV formula, see Thm.\ \ref{abbv formula}) all the Chern numbers must vanish. 
In this note we are interested
in the case in which the fixed point set is non-empty and discrete (for recent results concerning non-isolated fixed points
see \cite{Ku}). 

\emph{Henceforth, the triple $(\M,\J,S^1)$ will denote a compact, connected, almost complex manifold acted on by a circle $S^1$ that preserves $\J$, with nonempty, discrete fixed point set $\M^{S^1}$, and will be referred to as an} {\bf $S^1$-space}. 

The Chern numbers of $S^1$-spaces satisfy more restrictions. 
For instance, as a consequence of the ABBV formula, it is easy to see that 
$\cc_n[\M] = \chi(\M) = |\M^{S^1}|$,
thus implying that $\cc_n[\M]>0$, which is not true in general for any almost complex manifold $(\M,\J)$.
In 1979 Kosniowski \cite{Ko} conjectured that the number of fixed points, and hence $\cc_n[\M]$, grows linearly with $n$; more precisely
he predicted that $\cc_n[\M] \geq \left \lceil{\frac{n}{2}}\right \rceil$. Even if much progress has been done to prove Kosniowski's conjecture (see \cite{Ha}, and more recently \cite{LL,LT,PT,CKP,GPS,J}), a complete answer is still missing. This
 shows that the geography problem for an $S^1$-space $\ac$ is much harder, and
the following questions naturally arise:
\begin{question}\label{conj 1}
What are all the possible values of the Chern numbers of $\ac$? Are there other (combinations of) Chern numbers satisfying (in)equalities depending on $n$?
\end{question}
The first goal of this note is to show that the Chern numbers of $\ac$ satisfy equations that depend on two integers,
the \emph{index} $\k0$ of $(\M,\J)$, and an integer $N_0$ defined by the action, see below. The second is to apply these results to symplectic manifolds supporting symplectic circle actions with discrete fixed point set, 
showing `rigidity' results for the Chern numbers, and deriving topological conditions which ensure the manifold can only support Hamiltonian or only non-Hamiltonian actions, see Section \ref{atsm}.

 Let $\cc_1\in H^2(\M;\Z)$ be the first Chern class of the tangent bundle. The \emph{index} $\k0$ of $(\M,\J)$ is defined to be the largest integer such that, modulo torsion, $\cc_1=\k0\,\eta_0$ for some non-zero element $\eta_0\in H^2(\M;\Z)$. In other words, $\k0=0$ if $\cc_1$ is torsion, and is otherwise the biggest integer
 such that $\cc_1/\k0\in H^2(\M;\Z)$, modulo torsion elements.
When $\M$ is simply connected and symplectic, the index coincides with the \emph{minimal Chern number} (see Remark \ref{mcn}).
Note that $\M$ is simply connected if it is endowed with a Hamiltonian circle action with isolated fixed points, see \cite{Li2}.
The other integer $N_0$ depends on the action, and is defined as 
the \emph{number of fixed points with $0$ negative weights} (see Section \ref{background} \eqref{weights def}).  

When $\k0=0$, namely when $\cc_1$ is a torsion element, all the Chern numbers involving the first Chern class, as well as the Todd genus (see Lemma \ref{c1 N0} (a2)), must vanish.

 In this paper we are interested in
analysing what happens when $\k0>0$, and a careful analysis is carried out when $\k0\geq n-2$. 
When $\cc_1$ is not torsion, the aforementioned equations among the Chern numbers of $\ac$ are derived by analysing the zeros and the symmetries of the 
\emph{Hilbert polynomial} of $(\M,\J)$, which is defined as follows.
Let $\mathbb{L}_0\to \M$ be a line bundle whose first Chern class $\cc_1(\mathbb{L}_0)$ is $\eta_0=\frac{\cc_1}{\k0}$. 
Then the Hilbert polynomial $\Hil(z)$ is the polynomial in $\R[z]$ that, at integer values $k\in \Z$, gives the topological index of the bundle $\LL_0^k$, the $k$-tensor power of
$\LL_0$ (note that $\eta_0$ is only defined up to torsion, however $\Hil(z)$ does not depend on this choice, see Sect.\ \ref{equations chern}).
By the Atiyah-Singer formula, for every $k\in \Z$, the integer $\Hil(k)$ can be expressed in terms of Chern numbers of $\ac$:
\begin{equation}\label{H and c}
\Hil(k)=\left( \sum_{h=0}^n \frac{(k \,\eta_0)^h}{h!}\right)\ttot[\M]
= \left( \sum_{h=0}^n \frac{(k \,\cc_1)^h}{\k0^h\,h!}\right)\left( 1+\frac{\cc_1}{2}+ \frac{\cc_1^2+\cc_2}{12}+\cdots\right)[\M]
\end{equation}
where $\ttot=\sum_{j\geq 0}T_j= 1+\frac{\cc_1}{2}+ \frac{\cc_1^2+\cc_2}{12}+\cdots$ is the total Todd class of $\M$, and $T_j\in H^{2j}(\M;\Z)$ the Todd polynomials of $(\M,\J)$, for $j=0,\ldots,n$, namely the polynomials in the Chern classes of $(\M,\J)$ belonging to the power series $\frac{x}{1-e^{-x}}$. Note that in particular $\Hil(0)=T_n[\M]$, the Todd genus of $\M$, which in turn is equal to $N_0$ (Proposition \ref{properties P} (1)).

Using equivariant extensions of $\LL_0^k$ and localization in equivariant $K$-theory (the Atiyah-Segal formula \eqref{AS formula}),
it is proved that changing the orientation on $S^1$ implies the following
 `\emph{reciprocity law}' for $\Hil(z)$ (Propositions \ref{symmetries} and \ref{properties P} (2)):
\begin{equation}\label{reciprocity}
\Hil(z)=(-1)^n \Hil(-\k0-z)\,.
\end{equation}
This generalises, in the sense described in Sect.\ \ref{connections ehrhart}, a reciprocity law known for the Ehrhart polynomial of a reflexive polytope due to Hibi \cite{Hibi}.

The next theorem is the key result of Section \ref{equations chern}:
\begin{theorem}\label{main theorem}
Let $(\M,\J, S^1)$ be an $S^1$-space.  
Assume that the index $\k0$ of $(\M,\J)$ is greater or equal to $2$. Let $\Hil(z)$ be the associated Hilbert polynomial and $\deg(\Hil)$ its degree.
Then \\
\begin{align}\label{H=0 even}
 & \Hil(-1)=\Hil(-2)=\cdots = \Hil(-\k0+1)=0\,.
 \end{align}
 $\;$\\
Moreover,  if $\Hil(z)\not\equiv 0$, then 
\begin{equation}\label{bound k0}
 \k0\leq \deg(\Hil)+1\leq n+1\,.
\end{equation}
\end{theorem}
Equations \eqref{H and c} and \eqref{H=0 even} suggest that studying the Chern numbers of $\ac$ for large values of $\k0$ is easier.

In Sect.\ \ref{sec: generating fct} it is proved that, as a consequence of \eqref{reciprocity} and \eqref{H=0 even},
\emph{the number of conditions that determine the coefficients of $\Hil(z)$ is the same for $\k0=n+1-2k$ and $\k0=n-2k$, for every $k\in \Z$ such that $0\leq k \leq \frac{n-1}{2}$} (see Remark \ref{num of cds}). 
This follows from the fact that the generating function of $\Hil(z)$ is
 a rational function of the form $\Gen(t)=\mathrm{U}(t)/(1-t)^{\deg(\Hil)+1}$, where $\mathrm{U}(t)$ is a polynomial which ---up to a power of $t$---
 is \emph{self-reciprocal} or \emph{palindromic} (see Proposition \ref{gen fct hilbert} and Corollary \ref{U palindrom}). 

Using the results above we prove that, for $\k0\in \{n,\,n+1\}$, $\Hil(z)$ is completely determined by $N_0$; more precisely we prove that $\Hil(z)=N_0 \Hil_{\overline{M}}(z)$, the manifold $\overline{M}$ being $\C P^n$ for $\k0=n+1$ and the hyperquadric $Q_n$ in $\C P^{n+1}$ for $\k0=n$. This gives equations for the combinations of Chern numbers $\cc_1^h\, T_{n-h}[\M]$ in terms of $n$, for every $h=0,\ldots,n$, and in particular the values of $\cc_1^n[\M]$ and $\cc_1^{n-2}\cc_2[\M]$ (see Propositions \ref{cor n+1} and \ref{cor n}). 
   
When $\k0=n-1$ (and $n\geq 2$) or $\k0=n-2$ (and $n\geq 3$), $\Hil(z)$ and the combinations of Chern numbers $\cc_1^h\,T_{n-h}[\M]$, for $h=0,\ldots,n$, depend on a parameter. We compute explicitly their expressions in terms of this parameter (Propositions \ref{k0=n-1} and \ref{k0=n-2}) and determine a linear equation in $\cc_1^n[\M]$ and $\cc_1^{n-2}\cc_2[\M]$ which depends on $n$ and $N_0$: this is the content of Corollary \ref{relation c122} and Corollary \ref{relation c122 2}. 

Inter alia, we study the position of the roots of $\Hil(z)$ for $\k0\geq n-2$ and $\k0\neq 0$, making
connections with the work of Rodriguez-Villegas \cite{RV} and Golyshev \cite{Go}.
 
Finally, in Section \ref{examples} we investigate how in low dimensions the Chern numbers of $(\M,\J,S^1)$ depend on the integers $N_j$, for $j=0,\ldots,n$, defined as the number
of fixed points with $j$ negative weights. 
For instance, we prove that for $\k0=n$ or $n+1$, and $n\leq 4$, all the Chern numbers of $(\M,\J,S^1)$ can be expressed as linear combinations of  the $N_j$'s,
and when $n=2$ having $\k0=2$ or $3$ implies relations among the $N_j$'s. 

\medskip

\subsection{Applications to symplectic manifolds}\label{atsm} 
In order to apply the results we obtained for almost complex manifolds to symplectic manifolds,
let $\J\colon T\M \to T\M$ be an almost complex structure compatible with $\omega$, namely $\omega(\cdot, \J \cdot)$ is a Riemannian metric. Since the set of such structures
is contractible, we can define complex invariants of $T\M$, namely Chern classes and Chern numbers. 

Let $(\M,\omega)$ be a compact, connected symplectic manifold endowed with a symplectic circle action with isolated fixed points. \emph{Such a space is henceforth denoted by $(\M,\omega,S^1)$}. 
It follows that the $1$-form $\iota_{\xi^\#}\omega$ is closed; here $\xi^\#$ denotes the vector field generated by the circle action. 
If the $1$-form $\iota_{\xi^\#}\omega$ is \emph{exact} the action is said to be \emph{Hamiltonian}, otherwise we call it \emph{non-Hamiltonian}.
In the first case, if $\psi\colon \M\to \R$ is a function satisfying
$
\iota_{\xi^\#}\omega=-d\psi\,,
$
then $\psi$ is called a \emph{moment map} for the $S^1$-action. 

The first consequence of Theorem \ref{main theorem} in the symplectic category follows from the fact that,
if the action is Hamiltonian, $\Hil(z)$ can never be
identically zero (see Remark \ref{H Ham}), and the index coincides with the minimal Chern number (see Remark \ref{mcn}), leading to the following
\begin{corollary}\label{minimal chern ham}
Let $(\M,\omega)$ be a compact, connected symplectic manifold of dimension $2n$.
 If $(\M,\omega)$ supports a Hamiltonian $S^1$-action with isolated fixed points, then its minimal Chern number coincides with the index $\k0$, and the following inequalities hold $$1\leq \k0 \leq n+1.$$
\end{corollary}
 This result can be considered the analogue in the Hamiltonian category of a theorem of Michelsohn \cite[Cor.\ 7.17]{Mi}, which asserts that the index of a compact complex manifold admitting
 a K\"ahler metric with positive Ricci curvature is at most $n+1$. The same conclusion also holds if $(\M,\J)$ is a compact almost complex
 manifold which can be endowed with a quasi-ample line bundle; this result is due to Hattori \cite{Ha} and is discussed in Remark \ref{hattori rmk}.
 If a compact symplectic manifold can be endowed with a non-Hamiltonian circle action with isolated fixed points, then
 there are three possibilities for the index and the Hilbert polynomial (see Corollary \ref{bound on k0 s} and Remark \ref{rmk 1}).
 
There are plenty of examples of compact symplectic manifolds that can be
endowed with a Hamiltonian circle action with isolated fixed points. Until not so long ago, it was indeed believed that
every symplectic circle action with isolated fixed points would be
Hamiltonian.
This is sometimes also known as the `McDuff conjecture', and holds\footnote{For $n=1$, the only compact symplectic surface that can be endowed
with a symplectic circle action with isolated fixed points is the sphere, which is simply connected, hence the action is Hamiltonian. 
For $n=2$ the same conclusion holds by a result of McDuff in \cite{MD1}.
In the same paper the author also proves the existence of a six-dimensional compact symplectic manifold with a non-Hamiltonian action, but the fixed point set is not discrete.}
 for $n=1$ and $2$ \cite{MD1}, as well as in many other particular cases (see for instance \cite{Fe,Fr,Go1,Go2,L,Ono,TW,J}).
It is only very recently that Tolman announced the following striking result:
\begin{thm}[Tolman '15 \cite{T3}]\label{tolman 6}
There exists a non-Hamiltonian symplectic circle action with exactly 32 fixed points on a closed, connected, six-dimensional symplectic manifold $(\widetilde{M},\omega)$.
\end{thm}
This theorem implies the existence of a non-Hamiltonian symplectic circle action with discrete fixed point set for every $n\geq 3$: it is sufficient to take products $\widetilde{M}\times M$, where $M$ is a compact symplectic manifold endowed with a Hamiltonian circle action with $|M^{S^1}|<\infty$
(see also \cite[Cor.\ 1.2]{T3}, where $M=\C P^{n-3}$). However these products
give, so far, the only known examples of symplectic manifolds with non-Hamiltonian circle actions with discrete fixed point set, and the construction of new examples seems far from trivial. 
Thus we ask the following `weaker' question:
\begin{question}\label{q 2}
Let $(\M,\omega)$ be a compact, connected symplectic manifold. Are there topological conditions which imply that $(\M,\omega)$ can only support a Hamiltonian or only a non-Hamiltonian action?
\end{question}
The answer we give to Question \ref{q 2} is in terms of the Chern numbers of $(\M,\omega,S^1)$.
It is already known that if $\cc_1$ is torsion in $H^2(\M;\Z)$, the manifold cannot support any Hamiltonian circle action (see \cite[Prop.\ 4.3]{GPS}, or also Lemma \ref{Lemma:c1 not torsion}, and \cite[Lemma 3.8]{T2}). 
Thus the analysis we carry out to answer Question \ref{q 2} is under the hypothesis that $\cc_1$ is not torsion. 
A result of Feldman \cite{Fe} asserts that the Todd genus $T_n[\M]$ of $(\M,\omega,S^1)$ is either $1$ or $0$, and it is zero precisely if the action is non-Hamiltonian.
Although Feldman's result is very strong and gives an answer to Question \ref{q 2}, computing the Todd genus in high dimensions is difficult, since $T_n$ becomes a complicated combination of Chern classes. 
In some sense, our results can be regarded as a refinement of Feldman's, since we prove that
\emph{given a compact symplectic manifold $(\M,\omega)$, if certain combinations of Chern numbers vanish, then $(\M,\omega)$ cannot support any Hamiltonian circle action with isolated fixed points}. These combinations of Chern numbers depend on $\k0$, and are easier to compute than the Todd genus if $\k0$ is big enough (see Corollary \ref{cor non ham 2}). 
For $\k0\geq n-2$, we strengthen the result above by giving the possible values of $\cc_1^n[\M]$, $\cc_1^{n-2}\cc_2[\M]$ or a combination of them, these values depending on whether the action is Hamiltonian or not.
This is summarized in the following
 \begin{thm}[{\bf Hamiltonian vs non-Hamiltonian symplectic $S^1$-actions}]\label{nHam-char}
$\;$\\
\noindent
Let $(\M,\omega)$ be a compact, connected symplectic manifold, and suppose it can be endowed with a symplectic circle action with isolated fixed points.
Let $\k0$ be its index. Then:
\begin{itemize}
\item[(I)] If $\k0=0$ or $\k0 > n+1$ the action is non-Hamiltonian and $\cc_1^n[\M]=\cc_1^{n-2}\cc_2[\M]=0$.
\item[(II)] If $\k0=n+1$ then $(\cc_1^n[\M],\cc_1^{n-2}\cc_2[\M])$ is equal to $\Big((n+1)^n,\frac{n(n+1)^{n-1}}{2}\Big)$ or $(0,0)$. 
\item[(III)] If $\k0=n$ then $(\cc_1^n[\M],\cc_1^{n-2}\cc_2[\M])$ is equal to $\Big( 2n^n,n^{n-2}(n^2-n+2)\Big)$ or $(0,0)$.
\end{itemize}
Moreover, in \emph{(II)} and \emph{(III)} the
action is Hamiltonian if and only if $\cc_1^n[\M]\neq 0$ (or equivalently if and only if $\cc_1^{n-2}\cc_2[\M]\neq 0$).
\begin{itemize}
\item[(IV)] If $\k0=n-1$ and $n\geq 2$ then 
\begin{equation}\label{mm1}
\cc_1^{n-2}\cc_2[\M]-\frac{n(n-3)}{2(n-1)^2}\cc_1^n[\M]\quad \in \Big\{0,12 (n-1)^{n-2}\Big\}.
\end{equation}
\item[(V)] If $\k0=n-2$ and $n\geq 3$ then 
\begin{equation}\label{mm2}
\cc_1^{n-2}\cc_2[\M]-\frac{n-3}{2(n-2)}\cc_1^n[\M]\quad \in \Big\{0,24 (n-2)^{n-2}\Big\}.
\end{equation}
\end{itemize}
Moreover, in \emph{(IV)} (resp.\ \emph{(V)}) the action is Hamiltonian if and only if the combination of Chern numbers in \eqref{mm1} (resp.\ \eqref{mm2}) does not vanish.
\end{thm}
\begin{rmk}
\begin{itemize}
\item[(1)] This theorem implies that for $\k0\geq n-2$ the Chern numbers $\cc_1^n[\M]$ and $\cc_1^{n-2}\cc_2[\M]$ of $(\M,\omega,S^1)$ are very \emph{rigid}. Hence it gives necessary conditions for a compact, connected symplectic manifold $(\M,\omega)$ with $\k0> \max\{n-3,0\}$ to support a symplectic circle action with isolated fixed points. 
\item[(2)] Given a compact, connected symplectic manifold $(\M,\omega)$ of dimension $2n$, Theorem \ref{nHam-char} implies that
if the index satisfies $\k0\geq n$ and $\cc_1^n[\M]$ or $\cc_1^{n-2}\cc_2[\M]$ vanish, then $(\M,\omega)$ cannot be endowed with \emph{any} Hamiltonian circle
action with isolated fixed points. A similar conclusion holds for $\k0\in \{n-2,n-1\}$, by considering the combinations of Chern numbers in \eqref{mm1} and \eqref{mm2}.
\item[(3)] The above results are stated in terms of the Chern numbers $ \cc_1^n[\M]$ and $ \cc_1^{n-2}\cc_2[\M]$; however similar conclusions 
can be obtained for $\cc_1^h\,T_{n-h}[\M]$, for $h=0,\ldots,n$ (see Remark \ref{other comb}).
\end{itemize}
\end{rmk}
Finally, in Section \ref{examples} we analyse the geography problem for $(\M,\omega,S^1)$ when $n\leq 4$.
One of the goals is to find,
in the Hamiltonian case,
formulas for the Chern numbers in terms of $\k0$ and the Betti numbers of $\M$. 
For instance, the geography problem for $n=2$ can be completely solved (Corollary \ref{geo s}), and for $n=3,4$
we solve it for every $\k0\geq n$ (Propositions \ref{dim 6} and \ref{dim 8}).
As a byproduct of the investigation in dimension $8$, we prove that if a compact, connected symplectic manifold
of dimension $8$ supports a Hamiltonian $S^1$-action with
isolated fixed points, and if the minimal Chern number is even, then $\cc_2^2[\M]+2\, b_2(\M)=98 +b_4(\M)$ (Corollary \ref{c228h}).

\vspace{.5cm} 
\textbf{Acknowledgements.}\
First of all, I would like to thank Leonor Godinho for many fruitful conversations during the time I spent at Instituto Superior T\'ecnico, for inspiring this work and reading previous drafts. I would also like to thank Hansj\"org Geiges for useful discussions, and in particular for suggesting Remark \ref{mcn}. Frederik von Heymann explained to me many useful facts about reflexive polytopes, and strongly inspired Section 
\ref{connections ehrhart}.

Although I have never met him, I would like to dedicate this work to the memory of Akio Hattori who, through his articles, taught me so much. 

\section{Background and preliminary results}\label{background}
The main purpose of this section is to recall background material, set up notation and state preliminary results needed in the forthcoming sections.

Let $(\M,\J)$ be a compact, connected almost complex manifold of dimension $2n$.
Thus $\J\colon T\M \to T\M$ is a complex structure on the tangent bundle of $\M$, and 
for such manifold we consider the Chern classes of the tangent bundle, denoted by $\cc_j\in H^{2j}(\M;\Z)$\footnote{To avoid confusion, if in the same paragraph we also deal with
Chern classes of other bundles, we
will denote the Chern 
classes of the tangent bundle by $\cc_j(\M)$.},
as well as the Chern numbers 
$ \cc_{j_1}\cdots \cc_{j_l}[\M]\in \Z$, for every partition $(j_1,\ldots,j_l)$ of $n$, i.e.\ $j_1+\ \cdots \ +j_l=n$ and $j_m\in \mathbb{N}$ for $m=1\ldots,l$.

Moreover assume that $\ac$ is an $S^1$-space, i.e.\ $(\M,\J)$ is endowed with a $\J$-preserving $S^1$-action with nonempty and discrete fixed point set $\M^{S^1}=\{p_0,\ldots,p_N\}$, for some $N\in \Z_{>0}$.

For every $p_i\in M^{S^1}$ we denote by $w_{i,1},\ldots,w_{i,n}$ the \emph{weights} of the (isotropy) action of $S^1$ at $p_i$, i.e.\
the $S^1$ representation induced on $T_p\M$ is given by
\begin{equation}\label{weights def}
\alpha\cdot(z_1,\ldots,z_n)=(\alpha^{w_{i,1}}z_1,\ldots,\alpha^{w_{i,n}}z_n)\;\quad\mbox{for every}\quad \alpha\in S^1,
\end{equation}
for a suitable choice of complex coordinates $(z_1,\ldots,z_n)$ on $T_p\M\simeq \C^n$. We also denote by $W_i$ the (multi)set of weights at $p_i$, i.e.\;$W_i=\{w_{i,1},\ldots,w_{i,n}\}$.
Note that $w_{i,j}$ is nonzero for every $i=1\ldots,N$ and $j=1\ldots,n$, since the isotropy action commutes with the action on the manifold $\M$,
and $\M^{S^1}$ is discrete.
Finally, we denote by $\lambda_i$ the number of negative weights at $p_i\in M^{S^1}$ and by $N_j$ the number of fixed points with exactly $j$ negative weights, for every $j=0,\ldots,n$. From  \cite[Proposition 2.6]{Ha} we have that
\begin{equation}\label{NiN}
N_j=N_{n-j}\quad \mbox{for every}\quad j=0,\ldots,n\,. 
\end{equation}

Let $\K(\M)$ (resp.\;$\Ke(\M)$) be the ordinary (resp.\;$S^1$-equivariant) $K$-theory ring of $\M$, i.e.\;the abelian group associated to the
semigroup of isomorphism classes of complex vector bundles (resp.\;complex $S^1$-vector bundles) over $\M$,
endowed with the direct sum $\oplus$ and tensor product $\otimes$ operation.
Thus in particular
$\K(\{pt\})\simeq \Z$ and $\Ke(\{pt\}) \simeq R(S^1),$
the character ring of $S^1$. Henceforth, we identify the latter with the Laurent
polynomial ring $\Z[t,t^{-1}]$, where $t$ denotes the standard $S^1$-representation. 

Let $H_{S^1}^*(\M;\Z)$ be the $S^1$-equivariant cohomology of $\M$ with $\Z$ coefficients; we recall that this is defined to be the
ordinary cohomology of the Borel model, i.e.\;$
H_{S^1}^*(\M;\Z):=H^*(\M\times_{S^1}S^{\infty};\Z)\,,
$ 
where $S^{\infty}$ is the unit sphere in $\C^{\infty}$. Thus in particular $H_{S^1}^*(\{pt\};\Z)=\Z[x]$, where $x$ has degree $2$.

Finally, let $\pic(\M)$ (resp.\;$\pic_{S^1}(\M)$) be the Picard group of isomorphism classes of complex line bundles (resp.\;equivariant complex line bundles) over $\M$.

In the rest of the section, $\co(\cdot)$ (resp.\;$\coe(\cdot)$) will either denote the cohomology (resp.\;equivariant cohomology) ring 
with $\Z$ coefficients, the $K$-theory (resp.\;equivariant $K$-theory) ring, or the Picard (resp.\;equivariant Picard) group.

For $p\in \M^{S^1}$ let $i_p\colon \{p\}\hookrightarrow \M$ and $i\colon \M^{S^1}\hookrightarrow \M$ denote the natural inclusions; since they are equivariant we have the following induced maps:
$$
i_p^*\colon \coe(\M)\to \coe(\{p\})
$$
and
\begin{equation}\label{istar}
 i^*=\bigoplus_{p\in \M^{S^1}}i_p^*\colon \coe(\M)\to \coe(\M^{S^1})=\bigoplus_{p\in \M^{S^1}}\coe(\{p\})\;.
\end{equation}
We denote $i_p^*(K)$ simply by $K(p)$, for every $p\in \M^{S^1}$ and $K\in \coe(\M)$.

Observe that the unique map
$ \M\to \{pt\}$ induces maps  
\begin{equation*}
\coe(\{pt\})\to \coe(\M)\quad\text{and} \quad \co(\{pt\})\to \co(\M),
\end{equation*}
which give $\coe(\M)$ the structure of an $\coe(\{pt\})$-module, and $\co(\M)$ the structure of an $\co(\{pt\})$-module.

Finally, if $e$ denotes the identity element in $S^1$, the inclusion homomorphism 
$\{e\}\hookrightarrow S^1$ induces a restriction map, also called the ``forgetful homomorphism" 
\begin{equation}\label{restriction}
 r_{\mathcal{H}}\colon \coe(\M)\to \co(\M)\;.
\end{equation}
When $M$ is a point, $r_{\mathcal{H}}$ coincides with the evaluation at $x=0$ in cohomology, and 
 with the evaluation at $t=1$ in $K$-theory and in the Picard group. 
 The homomorphism \eqref{restriction} will be denoted by $r_H$ in cohomology, by $r_K$ in $K$-theory and by $r_{\pic}$ for the Picard group. 

\subsection{Indices of $K$-theory classes}\label{subsec: indeces}
Let 
\begin{equation}\label{indK}
 \ind\colon \K(\M)\to \K(pt)\simeq\Z
\end{equation}
 and 
 \begin{equation}\label{indKe}
 \ind_{S^1}\colon \Ke(\M)\to \Ke(pt)\simeq \Z[t,t^{-1}]  
 \end{equation}
 be the index homomorphisms (or $K$-theoretic push forwards) in ordinary and equivariant $K$-theory.
By the Atiyah-Singer formula, the index in \eqref{indK}  
can be computed as 
\begin{equation}\label{AT formula}
\ind(V)= \ch(V)\ttot [\M]\;,\quad\mbox{for every}\quad V\in \K(\M),
\end{equation}
where $\ch(\cdot)$ is the Chern character homomorphism $\ch\colon \K(\M)\to H^*(\M;\Q)$, and $\ttot$ is the total Todd class of $\M$, i.e.\;the cohomology 
class in $H^*(M;\Z)$ associated to the power series $\displaystyle\frac{x}{1-e^{-x}}$. This is a rational combination of Chern classes, and
the first terms of $\ttot$ are given by
\begin{equation}\label{Todd}
 \ttot=\sum_{j\geq 0}T_j=1+\frac{\cc_1}{2}+\frac{\cc_1^2+\cc_2}{12}+\frac{\cc_1\cc_2}{24}+\frac{-\cc_1^4+4\cc_1^2\cc_2+3\cc_2^2+\cc_1\cc_3-\cc_4}{720}+\ldots 
\end{equation}
where $T_j\in H^{2j}(\M;\Z)$ for every $j$. We also recall that the Todd genus $\td(\M)$ of $\M$ is given by
$$
\td(\M)= \ttot[\M]=T_n[\M]\;.
$$

By the Atiyah-Segal formula \cite{AS},  
the equivariant index \eqref{indKe} of a class $V\in \Ke(\M)$ can be computed 
in terms of $i^*(V)$ and the $S^1$ isotropy representation on $T\M\rvert_{\M^{S^1}}$. Since $M^{S^1}$ is discrete, the Atiyah-Segal formula in this case gives 
\begin{equation}\label{AS formula}
\ind_{S^1}(V)=\sum_{i=0}^N \frac{V(p_i)}{\prod_{j=1}^n (1-t^{-w_{i,j}})}\,,\quad\mbox{for every}\quad V\in \Ke(\M)\;.
\end{equation} 
By \eqref{AT formula}, \eqref{AS formula} and the commutativity of the following diagram 
\begin{equation}\label{K commutes}
\xymatrix{
\Ke(\M) \ar[r]^{r_\K} \ar[d]_{\ind_{S^1}} & \K(\M) \ar[d]_{\ind} \\
 \Z[t,t^{-1}] \ar[r]^{r_\K} &   \Z.
 } \
\end{equation}
it follows that for every $V\in \Ke(M)$ we have
\begin{equation}\label{formula index 2}
\left(\sum_{i=0}^N \frac{V(p_i)}{\prod_{j=1}^n (1-t^{-w_{i,j}})}\right)_{\rvert_{t=1}}=
r_\K(\ind_{S^1}(V))=\ind(r_\K(V))=
 \ch(r_\K(V))\ttot[\M]\;.
\end{equation}

We conclude this subsection by recalling the Atiyah-Bott-Berline-Vergne Localization formula \cite{At,BV}:
\begin{theorem}[ABBV Localization formula]\label{abbv formula}
Let $\M$ be a compact oriented manifold endowed with a smooth $S^1$-action.
Given $\mu\in H_{S^1}^*(\M;\Q)$
\begin{equation*}
\mu[\M]= \sum_{F}\frac{i_F^*(\mu)}{e^{S^1}(N_F)}[F]\;,
\end{equation*}
where the sum is over all the fixed-point set components $F$ of the action, and
$e^{S^1}(N_F)$ is the equivariant Euler class of the normal bundle to $F$.
\end{theorem} 
%%%%%

\subsection{Equivariant Chern classes and equivariant complex line bundles}\label{ecc}

Given a complex vector bundle $V\to \M$, denote by
$\cc(V)=\sum_i\cc_i(V)\in H^*(\M;\Z)$ the total Chern class of $V$, and if $V$ is equivariant, by $\cc^{S^1}(V)=\sum_i\cc_i^{S^1}(V)\in H^*_{S^1}(\M;\Z)$
the total equivariant Chern class, i.e.\;the total Chern class of the bundle
$V\times_{S^1}S^{\infty}\to \M\times_{S^1}S^{\infty}$.
It is easy to check that when $V=T\M$, if $\cc^{S^1}(\M)$ denotes the total equivariant Chern class 
of the tangent bundle $T\M$, then 
for every $p_i\in M^{S^1}$, $\cc^{S^1}(\M)(p_i)=\prod_{j=1}^n(1+w_{i,j}x)$, and hence
$\cc^{S^1}_j(\M)(p_i)=\sigma_j(w_{i,1},\ldots,w_{i,n})x^j$, where $\sigma_j(x_1,\ldots,x_n)$ denotes the $j$-th elementary polynomial in $x_1,\ldots,x_n$.

If $(\M,\J)$ is acted on by a circle $S^1$ preserving the almost complex structure,
it is a natural question to ask whether a given complex vector bundle $V$ over $\M$ admits an equivariant extension, i.e.\;whether the $S^1$-action can be
lifted to $V$, making the projection $V\to \M$ equivariant.  
This question has been studied in different settings, and 
for (complex) line bundles $\LL$ it has been completely answered by Hattori and Yoshida \cite[Theorem 1.1, Corollary 1.2]{HY} (see also \cite{HL,Mu} and \cite[Appendix C]{GKS}); 
here we summarise their main result in a different language.
\begin{theorem}[Hattori-Yoshida]
The equivariant first Chern class
\begin{equation}\label{isom ce}
\ce\colon \pic_{S^1}(\M)\to H_{S^1}^2(\M;\Z)
\end{equation} 
is an isomorphism. As a consequence,  
a line bundle $\LL$ admits an equivariant extension if and only if its first Chern class $\ce(\LL)$ is in the image
of the restriction map

\begin{equation}\label{restriction H2}
r_H\colon H_{S^1}^2(\M;\Z)\to H^2(\M;\Z). 
\end{equation}
\end{theorem}

The second assertion follows from the commutativity of the following diagram
$$
 \xymatrix{ 
\pic_{S^1}(\M) \ar[d]_{r_{\pic}} \ar[r]^-{\ce} &  H_{S^1}^2(\M;\Z)\ar[d]^{r_H} \\
\pic(\M) \ar[r]^-{\cc_1} & H^2(\M;\Z) \\
}
$$
and the fact that the first Chern class map $\cc_1$ on the bottom row is an isomorphism.

Moreover, for any line bundle $\LL$ whose first Chern class is in the image of \eqref{restriction H2}, which will henceforth be called
\emph{admissible},
all the possible equivariant
extensions are parametrised by $H^2(\C P^{\infty};\Z)\simeq \Z$. More precisely, given an admissible $\LL$ and two equivariant
extensions $\LL^{S^1}_1$ and $\LL^{S^1}_2$, there exists $a\in \Z$ such that $\ce(\LL^{S^1}_1)-\ce(\LL^{S^1}_2)=ax$. In particular we have that 
\begin{equation}\label{trivial constant}
\mbox{\emph{if}}\;\;\LL \;\;\mbox{\emph{is trivial, then }}\ce(\LL^{S^1})(p)=ax\quad\mbox{\emph{for every} }p\in \M^{S^1},\mbox{ \emph{for some} }a\in \Z.
\end{equation}

In \cite[Lemma 3.2]{Ha}, Hattori proves that if $\LL$ is admissible, and $\LL'$ is such that $\cc_1(\LL)=k\cc_1(\LL')$ for some nonzero integer $k$,
then $\LL'$ is also admissible; moreover every line bundle whose first Chern class is in $\tor(H^2(\M;\Z))$, the torsion subgroup of $H^2(\M;\Z)$, is admissible.
An example of admissible line bundle is given by the determinant line bundle $\Lambda^n(T\M)$.
In fact it is well-known that $\cc_1(\M)$ always admits an equivariant extension, given by the equivariant
first Chern class $\ce(\M)$. Hence $\Lambda^n(T\M)$ is admissible, since $\cc_1(\Lambda^n(T\M))=\cc_1(\M)$.
Moreover the trivial bundle is clearly admissible.

Let $\Lat$ be the lattice given by $H^2(\M;\Z)/\tor(H^2(\M;\Z))$ and $$\pi\colon H^2(\M;\Z)\to \Lat$$ the projection. The following lemma is an immediate consequence of \cite[Lemma 3.2]{Ha}.
\begin{lemma}\label{line admissible}
Let $\ac$ be an $S^1$-space and let $\cc_1$ be the first Chern class of the tangent bundle. 
Suppose that $\cc_1$ is not a torsion element, i.e.\;$\pi(\cc_1)\neq 0$, and let $\eta$ be a primitive element in $\Lat$ such that  $\pi(\cc_1)=k_0\eta$, for some $k_0\in \Z\setminus\{0\}$. Then every line bundle $\LL$ 
such that $\pi(\cc_1(\LL))=k\,\eta$ is admissible, for every $k\in \Z$.
\end{lemma}

Observe that the {\bf index} $\k0$ of $(\M,\J)$, as defined in the introduction, is the same as the largest integer satisfying $\pi(\cc_1)=\k0 \pi(\eta_0)$, for some non-torsion $\eta_0\in H^2(\M;\Z)$. Note that, when $\cc_1$ is not torsion, $\pi(\eta_0)$ is necessarily primitive in $\Lat$.

In the rest of this note, we will make use of the following {\bf convention}:
Let $\tau$ be an element of $H_{S^1}^2(\M^{S^1};\Z)$; thus $\tau(p)=a_px\in H_{S^1}^2(\{p\};\Z)$, where $a_p\in \Z$ and $x$ is the generator of $H_{S^1}^2(\{p\};\Z)=H^2(\C P^{\infty};\Z)$.
For the sake of simplicity, \emph{we henceforth identify $\tau\in H_{S^1}^2(\M^{S^1};\Z)$ with the map from $M^{S^1}$ to $\Z$ which assigns to $p$ the integer $a_p$.}
$\;$\\

Note that for every $\LL^{S^1}\in \pic_{S^1}(\M)$ and every $p_i\in M^{S^1}$ 
\begin{equation}\label{elb}
\LL^{S^1}(p_i)=t^{a_i},\quad \mbox{where}\;\; a_i \;\; \mbox{is the integer given by}\;\;\; \ce(\LL^{S^1})(p_i).
\end{equation}%

In virtue of the isomorphism \eqref{isom ce}, given a class $\tau\in H_{S^1}^2(\M;\Z)$ (resp.\;$\tau'\in H^2(\M;\Z)$), we will denote by $\e{\tau}$ the isomorphism class of equivariant line
bundles whose first equivariant Chern class is $\tau$ (resp.\;the isomorphism class of line
bundles whose first Chern class is $\tau'$).
We conclude this section with the following
\begin{prop}\label{symmetries}
Let $\ac$ be an $S^1$-space with
$\M^{S^1}=\{p_0,\ldots,p_N\}$. Let $\cc_1$ and $\ce$ be respectively the first Chern class and the equivariant first Chern class of the tangent bundle
of $\M$. Then, for every $\tau\in H_{S^1}^2(\M;\Z)$ we have
\begin{equation}\label{eq index symmetry}
\ind_{S^1}(\e{\tau})=(-1)^n\ind_{\widetilde{S}^1}(\e{(-\tau-\cc_1^{\widetilde{S}^1})})\,,
\end{equation}
where $\widetilde{S}^1$ is the circle $S^1$ with orientation reversed. 
Thus
\begin{equation}\label{index symmetry}
\ind(\e{r_H(\tau)})=(-1)^n\ind(\e{(-r_H(\tau)-\cc_1)}).
\end{equation}
\end{prop}
\begin{proof}
By \eqref{AS formula} and \eqref{elb} we have that
$$
\ind_{S^1}(\e{\tau})=\sum_{i=0}^N \frac{t^{\tau(p_i)}}{\prod_{j=1}^n(1-t^{-w_{i,j}})}=\sum_{i=0}^N\frac{(-1)^n\,t^{\tau(p_i)+w_{i,1}+\ldots+w_{i,n}}}{\prod_{j=1}^n(1-t^{w_{i,j}})}=(-1)^n\ind_{\widetilde{S}^1}(\e{(-\tau-\cc_1^{\widetilde{S}^1})})\,,
$$
and \eqref{index symmetry} follows from \eqref{K commutes}, \eqref{eq index symmetry} and the fact that $r_H(\ce)=r_H(\cc_1^{\widetilde{S}^1})=\cc_1$.
\end{proof}

\section{Computation of equivariant indices}\label{cei}

In this section we analyse some properties of the equivariant index
of an equivariant line bundle $\LL^{S^1}$. In particular we study  
under which conditions $\LL^{S^1}$ is `\emph{rigid}', namely when its equivariant index $\ind_{S^1}(\LL^{S^1})$ is 
$S^1$-invariant, i.e.\ it belongs to $\Z\subset \Z[t,t^{-1}]$, and determine what the constant is in terms of the restriction to the fixed points of its equivariant first Chern class: this is the
content of Theorem \ref{trick}. As a consequence, we derive conditions that ensure the equivariant index of an equivariant line bundle to be zero. 
This is a generalisation of arguments which had
already been used in different ways by several authors, see for example Hattori \cite[Proposition 2.6]{Ha}, Hirzebruch et al.\;\cite[Section 5.7]{Hi}, Li \cite{L} and 
Li-Liu \cite[Proposition 2.5]{LL}.

The rest of the section is devoted to deriving applications of Theorem \ref{trick} which will be used in the forthcoming sections.

For every point $p_i\in M^{S^1}$, we order the isotropy weights $w_{i,1},\ldots,w_{i,n}$ at $p_i$ in such a way that the first $\lambda_{i}$ are exactly the negative weights at $p_i$.
We define $\cp$ and $\cm$ in $H_{S^1}^2(\M^{S^1};\Z)$ to be
\begin{equation}\label{cpcm}
\cp(p_i)=w_{i,\lambda_{i}+1}+\cdots +w_{i,n}\;\;\;\quad \mbox{and} \quad \;\;\;\cm(p_i)=-(w_{i,1}+\cdots+w_{i,\lambda_i})\,. 
\end{equation}
From the definition it follows that $\cp(p_i)\geq 0$ (resp.\;$\cm(p_i)\geq 0$) and equality holds if and only if $\lambda_i=n$ (resp.\;$\lambda_i=0$).
Moreover, if $\ce$ denotes the equivariant first Chern class of $\M$, we have that $i^*(\ce)=\cp-\cm$. 

\begin{defin}
A class $\tau\in H_{S^1}^2(\M;\Z)$ is said to be \emph{dominated} by $\cp$  (resp.\;by $\cm$) if $\tau(p)\leq \cp(p)$ for every $p\in \M^{S^1}$ 
(resp.\;if $-\tau(p)\leq \cm(p)$ for every $p\in \M^{S^1}$). 
\end{defin}
\begin{rmk}\label{ex 0 and c1}
It is easy to check that the classes $\mathbf{0}$ and $\ce$ are always dominated by both $\cp$ and $\cm$. 
Moreover, if $\tau\in H_{S^1}^2(\M;\Z)$ satisfies $\tau(p)\leq 0$ (resp.\;$\tau(p)\geq 0$) for every $p\in \M^{S^1}$ then $\tau$ is dominated by $\cp$ (resp.\;$\cm$). 
\end{rmk}
\begin{theorem}\label{trick}
Let $\ac$ be an $S^1$-space with
$\M^{S^1}=\{p_0,\ldots,p_N\}$. Let $\tau$ be an element of $H_{S^1}^2(\M;\Z)$ and $\cp$, $\cm$ defined as above. 
For every $p\in \M^{S^1}$, define $\delta^+(p)$ (resp.\;$\delta^-(p)$) to be $1$ if $\tau(p)=\cp(p)$ (resp.\;$-\tau(p)=\cm(p)$) and zero otherwise.
Then
\begin{itemize}
 \item[(i)] If $\tau\in H_{S^1}^2(\M;\Z)$ is dominated by $\cp$ then 
 $$
 \ind_{S^1}(\e{(-\tau)})=\sum_{j\geq 0}b_jt^j\in \Z[t],\quad \mbox{and}\quad b_0= \sum_{i=0}^N\delta^+(p_i)(-1)^{n-\lambda_i}
 $$
 \item[(ii)] If $\tau\in H_{S^1}^2(\M;\Z)$ is dominated by $\cm$ then 
 $$
 \ind_{S^1}(\e{(-\tau)})=\sum_{j\leq 0}b_jt^j\in \Z[t^{-1}],\quad \mbox{and}\quad b_0= \sum_{i=0}^N\delta^-(p_i)(-1)^{\lambda_i}
 $$
 \item[(iii)]  If $\tau\in H_{S^1}^2(\M;\Z)$ is dominated by $\cp$ and $\cm$ then 
 \begin{equation}\label{index integer}
  \ind_{S^1}(\e{(-\tau)})=b_0\in \Z
 \end{equation}
where 
 \begin{equation}\label{index precise}
b_0=\sum_{i=0}^N\delta^+(p_i)(-1)^{n-\lambda_i}= \sum_{i=0}^N\delta^-(p_i)(-1)^{\lambda_i}.
 \end{equation}
\end{itemize}

\end{theorem}
\begin{proof}
By \eqref{AS formula} and \eqref{elb}, we have that 
 \begin{equation}\label{ei1}
 \ind_{S^1}(\e{(-\tau)}) =\sum_{i=0}^N \frac{   t^{-\tau(p_i)} }{\prod_{j=1}^n(1-t^{-w_{i,j}}) }
\end{equation} 
For every $i=0,\ldots,N$, let $f_i(t)$ be the rational function $\displaystyle \frac{   t^{-\tau(p_i)} }{\prod_{j=1}^n(1-t^{-w_{i,j}}) }$, and observe that
$\sum_{i=0}^Nf_i(t)\in \Z[t,t^{-1}]$.
Thus, in order to prove (i), it is sufficient to prove that $\lim_{t\to 0}\sum_{i=0}^Nf_i(t)$ is finite, and its value will be equal to $b_0$. Observe that by definition
of $\cp$, $f_i(t)$ can be rewritten as $\displaystyle \frac{(-1)^{n-\lambda_i}\;t^{-\tau(p_i)+\cp(p_i)}}{\prod_{j=1}^n(1-t^{|w_{i,j}|})}$.
Since by assumption $i^*(\tau)$ is dominated by $\cp$, $\lim_{t \to 0}f_i(t)$ is finite for all $i=0,\ldots,N$, and by definition of $\delta^+$ it follows that 
its value equals to $\delta^+(p_i)(-1)^{n-\lambda_i}$, thus proving (i).

The proof of (ii) follows by a similar argument, by taking $\lim_{t\to \infty}\sum_{i=0}^Nf_i(t)$, and by observing that $f_i(t)$ can be written as
$ \displaystyle \frac{(-1)^{\lambda_i}\;t^{-\tau(p_i)-\cm(p_i)}}{\prod_{j=1}^n(1-t^{-|w_{i,j}|})}$.

Finally, (iii) follows from (i) and (ii).
\end{proof}

\begin{exm}\label{exm:CP3}
Consider $(\C P^3,\J)$ with the standard (almost) complex structure, and $S^1$-action given by 
$$
\lambda \cdot [z_0:z_1:z_2:z_3]=[z_0:\lambda^a z_1:\lambda^{a+b}z_2:\lambda^{a+b+c}z_3],
$$
where $a,b,c$ are pairwise coprime positive integers. This action is ``standard'', in the sense that it is the
restriction to a subtorus of dimension $1$ of the standard toric action of the $3$-dimensional torus $\T^3$ on $\C P^3$.
The fixed point set is given by four points $p_0,p_1,p_2,p_3$, corresponding respectively to $[1:0:0:0],[0:1:0:0],[0:0:1:0],[0:0:0:1]$. 
Let $\tau_0$ be the generator of $H^2(\C P^3,\Z)$ such that $\cc_1(\C P^3)=4\, \tau_0$. It can be checked that $\tau_0$ admits an equivariant
extension\footnote{Indeed, in this case, every class $\gamma\in H^j(\C P^3,\Z)$ admits an equivariant extension, for every $j$. This is due to the fact
that $\C P^3$ with the above $S^1$-action is \emph{equivariantly formal} (see for example \cite{Ki}).} $\tau\in H^2_{S^1}(\C P^3,\Z)$, i.e.\;$r_H(\tau)=\tau_0$; we pick
$\tau$ so that $\tau(p_0)=0$. 
The (multi)sets of isotropy weights at each fixed point, as well as $i^*(\tau)$, $\cp$ and $\cm$, are given in the following table:
\begin{center}
\begin{tabular}{|l|| l|l|l|l|}
\hline
        & $\;\;\;\;\;\;\;\;\;\;\;\;\;\;\;\;W_i$ & $\;\;\;\;\;i^*(\tau)$ & $\;\;\;\;\;\;\;\cp$ & $\;\;\;\;\;\;\;\cm$ \\ \hline 
$p_0:$   & $\{a,a+b,a+b+c\}$ &  $0$ & $3a+2b+c$ & $0$ \\ \hline
$p_1:$   & $\{-a,b,b+c\}$ &  $-a$  & $2b+c$ & $a$ \\ \hline 
$p_2:$    & $\{-b,-a-b,c\}$  & $-a-b$  & $c$ & $a+2b$ \\ \hline 
$p_3$   & $\{-c,-b-c,-a-b-c\}$  & $-a-b-c$   & $0$ & $a+2b+3c$ \\ \hline 
\end{tabular}
\end{center}  
Observe that $\tau$ is dominated by both $\cp$ and $\cm$, and by definition $\delta^+\equiv 0$. Thus Theorem \ref{trick} (iii) implies that $\ind_{S^1}(\e{(-\tau)})=0$, as it
can also be checked directly from here 
\begin{align*}
\ind_{S^1}(\e{(-\tau)})= &\frac{1}{(1-t^{-a})(1-t^{-a-b})(1-t^{-a-b-c})}+\frac{t^a}{(1-t^{a})(1-t^{-b})(1-t^{-b-c})}\\
 &+\frac{t^{a+b}}{(1-t^{b})(1-t^{a+b})(1-t^{-c})}+
\frac{t^{a+b+c}}{(1-t^{c})(1-t^{b+c})(1-t^{a+b+c})}=0\\
\end{align*}

\end{exm}

\begin{rmk}\label{index positive or negative}
Following the discussion in Remark \ref{ex 0 and c1}, by Theorem \ref{trick} we have that if $\tau\in H_{S^1}^2(\M;\Z)$ satisfies $\tau(p)\geq 0$ (resp.\;$\tau(p)\leq 0$)
for all $p\in \M^{S^1}$, then $\ind_{S^1}(\e{\tau})\in \Z[t]$ (resp.\;$\ind_{S^1}(\e{\tau})\in \Z[t^{-1}]$).
\end{rmk}

As an immediate consequence of Theorem \ref{trick}, we have the following
\begin{corollary}\label{index 0 and -c1}
Let $\ac$ be an $S^1$-space with
$\M^{S^1}=\{p_0,\ldots,p_N\}$. Let $N_i$ be the number of fixed points with exactly $i$ negative weights.

If $\mathbf{1}\in \pic_{S^1}(\M)$ denotes the trivial line bundle over $\M$, where $\ce(\mathbf{1})=\mathbf{0}$,  then 
\begin{equation}\label{index 0}
\ind_{S^1}(\mathbf{1})=N_0=N_n\;.
\end{equation}
If  $\widetilde{\LL}^{S^1}\in \pic_{S^1}(\M)$ denotes the determinant line bundle $\Lambda^n(T^*\M)$, where $\ce(\widetilde{\LL}^{S^1})=\ce(\Lambda^n(T^*\M))=-\ce$, then
\begin{equation}\label{index -c1}
\ind_{S^1}(\widetilde{\LL}^{S^1})=(-1)^nN_0=(-1)^nN_n\;.
\end{equation}
\end{corollary}

\begin{proof}
As we have already remarked, the classes $\mathbf{0}$ and $\ce$ are dominated by $\cp$ and $\cm$. Thus \eqref{index 0}
and \eqref{index -c1} follow from Theorem \ref{trick} (iii) and the definition of $N_0$ and $N_n$. 
\end{proof}
Note that equation \eqref{index 0} is already known, see for example \cite[Corollary 2.7]{Ha} (also see \cite[Theorem 2.3]{L}). 

Observe that \eqref{index -c1} can also be obtained by noticing that since $\ce$ is dominated by $\cp$ and $\cm$, 
\eqref{index integer} implies that  
$\ind_{S^1}(\widetilde{\LL}^{S^1})$ is an integer, thus $\ind_{S^1}(\widetilde{\LL}^{S^1})=\ind(r_K(\widetilde{\LL}^{S^1}))$, 
and so \eqref{index -c1} follows from \eqref{index symmetry} in Proposition \ref{symmetries} and \eqref{index 0}.

We also remark that $\ind_{S^1}(\mathbf{1})$ is the Todd genus of $\M$; in fact
from  \eqref{formula index 2} we have that 
\begin{equation}\label{todd genus}
\td(\M)= T_n[\M]= \ch(r_K(\mathbf{1})) \ttot[\M]=\ind(r_K(\mathbf{1}))=\ind_{S^1}(\mathbf{1})\,
\end{equation}
where the second equality follows from observing that $\ch(r_K(\mathbf{1}))=1$, and the
last equality follows from \eqref{K commutes} and the fact that $\ind_{S^1}(\mathbf{1})$ is an integer, thus $\ind(r_K(\mathbf{1}))=r_K(\ind_{S^1}(\mathbf{1}))=\ind_{S^1}(\mathbf{1})$.
By combining \eqref{index 0} and \eqref{todd genus} we recover the following well-known fact (see \cite[Remark 2.10]{Ha} and \cite{Fe}).
\begin{corollary}\label{todd genus comp}
Let $\ac$ be an $S^1$-space,
$N_i$ the number of fixed points with exactly $i$ negative weights, and $\td(\M)$ the Todd genus of $\M$. Then
$$
\td(\M)=N_0=N_n.
$$
\end{corollary}
Before giving the main application of Theorem \ref{trick}, we prove the following easy but useful lemma. 
\begin{lemma}\label{c1 N0}
Let $\ac$ be an $S^1$-space, $\cc_1$ the first Chern class of the tangent bundle of $\M$, $N_i$ the number of fixed points with exactly $i$ negative weights, and $\td(\M)$ the Todd genus of $\M$.
\begin{itemize}
 \item[(a1)] If $\eta\in \tor(H^2(\M,\Z))$ then 
\begin{equation}\label{index torsion}
\ind(\e{\eta})=\td(\M)=N_0
\end{equation}
and
 \begin{equation}\label{torsion index}
  \ind_{S^1}(\e{\eta^{S^1}})=t^{a}\td(\M)=t^{a}N_0\,,
 \end{equation}
 where $\eta^{S^1}\in H_{S^1}^2(\M,\Z)$ denotes an equivariant extension of $\eta$, and $a=\eta^{S^1}(p)$ for every $p\in \M^{S^1}$.\\
\item[(a2)] If $\cc_1\in \tor(H^2(\M,\Z))$ then $N_0=N_n=0$ and $\td(\M)=0$.
\end{itemize}

\end{lemma}
\begin{proof}
(a1) First of all, observe that if $\eta\in \tor(H^2(\M,\Z))$ then, by the discussion in Section \ref{ecc}, it admits an equivariant extension $\eta^{S^1}\in H^2_{S^1}(\M,\Z)$. 
By the commutativity of \eqref{K commutes}, in order to prove \eqref{index torsion} it is sufficient to prove \eqref{torsion index}.
If $\eta$ is torsion then there exists $k\in \Z\setminus\{0\}$ such that $k\eta=0$. Thus if we consider an equivariant extension 
$\eta^{S^1}$, by \eqref{trivial constant} we have that $\eta^{S^1}(p)=a$ for some $a\in \Z$, for every $p\in \M^{S^1}$. Hence  
$$
\ind_{S^1}(\e{\eta^{S^1}})=t^{a}\ind_{S^1}(\mathbf{1})=t^{a}\td(\M)=t^aN_0
$$
where the first equality follows from \eqref{AS formula},
the second from \eqref{todd genus}, and the last from Corollary \ref{todd genus comp}.

(a2)
By a similar argument, we have that the integer 
 $\ce(p)$ does not depend on $p\in M^{S^1}$. However $\ce(p_i)=\sum_{j=1}^nw_{i,j}$, and by \eqref{NiN} we have $N_0=N_n$.
So by definition of $N_0$ and $N_n$ we must have that
$N_0=N_n=0$, and by Corollary \ref{todd genus comp} that $\td(\M)=0$.

\end{proof}

The next proposition also follows from Theorem \ref{trick}, but it is a key result for the theorems in the next sections (see also \cite[Assertion 4.10]{Ha} and \cite[Proposition 2.5]{LL}).
\begin{prop}\label{eq index zero}
Let $\ac$ be an $S^1$-space. Let $\ce$ be the equivariant first Chern class of the tangent bundle of $\M$ and $k$ a positive integer such that 
$\ce(p)=k\,\eta^{S^1}(p)+c$ for all $p\in \M^{S^1}$, for some $\eta^{S^1}\in H_{S^1}^2(\M;\Z)$ and $c\in \Z$.
Then
\begin{equation}\label{index 0 1..k0}
 \ind_{S^1}(\e{(-h\eta^{S^1})})=0\quad\mbox{for every}\quad h=1,\ldots,k-1\;.
\end{equation}

\end{prop}
\begin{rmk}
Observe that if $\cc_1$ is torsion then $r_H(\eta^{S^1})$ is also torsion, and by Lemma \ref{c1 N0} it follows that   
 \begin{equation}\label{index 0 always}
  \ind_{S^1}(\e{(-h\eta^{S^1})})=0\quad\mbox{for every}\quad h\in \Z
 \end{equation}
 \end{rmk}
\begin{proof}[Proof of Proposition \ref{eq index zero}]
First of all, observe that it is not restrictive to assume that $c=0$. In fact,
let $S^1\times \M\to \M$, $(\lambda,q)\to \lambda\cdot q$ be the given $S^1$-action on $\M$, and consider 
a new action given by $(\lambda,q)\to \lambda^k\cdot q$; we denote by $\widetilde{S}^1$ the new circle acting on $\M$.
Note that the set of fixed points of this action coincides with the old one,
and the new isotropy weights are the old ones multiplied by $k$. Thus 
 $\cc_1^{\widetilde{S}^1}(p)$ is divisible by $k$, for every $p\in \M^{\widetilde{S}^1}=\M^{S^1}$.
So there exists $\widetilde{\eta}\in H_{\widetilde{S}^1}^2(\M;\Z)$ such that
$\cc_1^{\widetilde{S}^1}=k\widetilde{\eta}$. Moreover, if $\ind_{S^1}(\e{\eta^{S^1}})=P(t,t^{-1})$ for some $P\in \Z[x,y]$, then 
$\ind_{\widetilde{S}^1}(\e{\widetilde{\eta}})=t^bP(t^k,t^{-k})$, for some $b\in \Z$. 
Thus $\ind_{S^1}(\e{\eta^{S^1}})=0$ if and only if $\ind_{\widetilde{S}^1}(\e{\widetilde{\eta}})=0$.
Hence we can assume that $\ce(p)=k\,\eta^{S^1}(p)$ for all $p\in \M^{S^1}$.

Notice that for all $p\in \M^{S^1}$ such that $\eta^{S^1}(p)>0$  and all $h=1,\ldots,k-1$, we have
\begin{equation}\label{keta}
h\,\eta^{S^1}(p)<k\,\eta^{S^1}(p)=\ce(p)=\cp(p)-\cm(p)\leq \cp(p)\,,
\end{equation}
thus $h\,\eta^{S^1}$ is dominated by $\cp$ for all $h=1,\ldots,k-1$. Moreover \eqref{keta} implies that
$\delta^+(p)=0$ for all $p\in \M^{S^1}$ such that $\eta^{S^1}(p)>0$, and since $\cp(p)$ is always nonnegative,
$\delta^+(p)=0$ for all $p\in \M^{S^1}$ such that $\eta^{S^1}(p)\neq 0$. 
Finally observe that if $\eta^{S^1}(p)=\cp(p)=0$, then $\ce(p)=0$ and $\cm(p)=0$; however this is impossible, unless
$\dim(\M)=0$. So we can conclude that $\delta^+(p)=0$ for all $p\in \M^{S^1}$.

A similar argument shows that $h\,\eta^{S^1}$ is dominated by $\cm$ for all $h=1,\ldots,k-1$ (and $\delta^-(p)=0$ for all $p\in \M^{S^1}$).
So the conclusion follows from Theorem \ref{trick} (iii).  
\end{proof}

\subsection{Symplectic manifolds} Suppose that $(\M,\omega)$ is a compact, connected symplectic manifold endowed with a symplectic circle action with isolated fixed points. We recall that this triple is denoted by $(\M,\omega,S^1)$. 
The following lemma is a key fact to translate our results in the almost complex category to the symplectic category.
\begin{lemma}[\cite{MD1}]\label{N0 1}
Given $(\M,\omega,S^1)$, 
then $N_0$ can be either $0$ or $1$, and is $1$ exactly if the action is Hamiltonian.
\end{lemma}
 If the action is Hamiltonian, then $N_0$ coincides indeed with the number of points of minima of the moment map $\psi$, which
 is $1$ because $\psi$ is a Morse function with only even indices, and $\M$ is assumed to be connected. 
More in general, the equivariant perfection of $\psi$ (see \cite{Ki}) implies that
\begin{equation}\label{bi=Ni}
b_{2j}(\M)= N_j \quad \mbox{for every}\quad j=0,\ldots,n\,, 
\end{equation}
where $b_{2j}(\M)$ denotes the $2j$-th Betti number of $\M$. The following fact is a consequence of the results of this section:
\begin{lemma}\label{Lemma:c1 not torsion}
Given $(\M,\omega,S^1)$, if the action is Hamiltonian then $\cc_1$ is not a torsion
element in $H^2(\M;\Z)$. 
\end{lemma}
\begin{proof}
It is sufficient to combine Lemma \ref{N0 1} with Lemma \ref{c1 N0} (a2). 
\end{proof}
\begin{rmk}\label{mcn}
Let $(\M,\omega)$ be a compact symplectic manifold with first Chern class $\cc_1$, and suppose it is not torsion.
Following Definition 6.4.2 in \cite{MDS}, the \emph{minimal Chern number} of $(\M,\omega)$ is defined to be the integer $N$ such that $\langle \cc_1,\pi_2(\M)\rangle = N \Z$.
If $\M$ is simply connected then, by the Hurewicz theorem, we have $\pi_2(\M)=H_2(\M,\Z)$ which, modulo torsion, is isomorphic to $H^2(\M,\Z)$, thus implying that the minimal Chern number agrees with the index of $(\M,\omega)$. A result of Li \cite{Li2} implies that 
 if the $S^1$-action on $(\M,\omega)$ is Hamiltonian with isolated fixed points then $\M$ is simply connected. So it follows that
 if $(\M,\omega)$ is endowed with a Hamiltonian $S^1$-action with isolated fixed points, the minimal Chern number always agrees with the index $\k0$, which is not zero by Lemma \ref{Lemma:c1 not torsion}. \end{rmk}
%%%%%%%%%%%%%%%%%%%%%%%%%%%%%%%%%%%%%%%%%%%%%%%%%%%%%%%%%%

\section{The Hilbert polynomial of $(\M,\J)$ and the equations in the Chern numbers}\label{equations chern}
We recall from Section \ref{ecc} that $\mathcal{L}$ is the lattice given by $H^2(\M;\Z)/\tor(H^2(\M;\Z))$ and $\pi$ the projection $\pi\colon H^2(\M;\Z)\to \mathcal{L}$.
If $\cc_1$ is not torsion we have $\pi(\cc_1)\neq 0$, so there exists a non-torsion element $\eta_0\in H^2(\M;\Z)$ such that $\pi(\cc_1)=\k0\,\pi(\eta_0)$.
The index $\k0$, and when $\k0>0$ the associated $\eta_0\in H^2(\M;\Z)$ (uniquely defined up to torsion), will play a crucial role in the rest of the section.

Before proceeding, we prove the following Lemma:
\begin{lemma}\label{index torsion independent}
Let $\eta\in H^2(\M;\Z)$ and $\tau\in \tor(H^2(\M;\Z))$. Then  
$$
\ind(\e{(\eta+\tau)})=\ind(\e{\eta})\,.
$$
\end{lemma}
\begin{proof}
By \eqref{AT formula} we have that 

\begin{align*}
\ind(\e{(\eta+\tau)})& = \ch(\e{(\eta+\tau)})\ttot[\M]= \left(1+\eta+\frac{\eta^2}{2}+\cdots\right)\left(1+\tau+\frac{\tau^2}{2}+\cdots\right)\ttot[\M]=\\
 & = \left(1+\eta+\frac{\eta^2}{2}+\cdots\right)\ttot[\M]=\ind(\e \eta)\,,
\end{align*}
where the second-last equality follows from the fact that if $\tau$ is torsion then $ \tau^k\alpha[\M]=0$ for all $k>0$ and $\alpha\in H^{2n-2k}(\M;\Z)$.
\end{proof}

In the rest of the section \emph{we assume that $\cc_1$ is not torsion}. Let $\eta_0\in H^2(\M;\Z)$ be such that $\pi(\cc_1)=\k0 \pi(\eta_0)$. Even if $\eta_0$
is not uniquely defined, by Lemma \ref{index torsion independent} the topological index $\ind(\e{\eta})$ is independent on $\eta\in \pi^{-1}(\pi(\eta_0))$. 
Hence, given $(\M,\J)$ with $\cc_1$ not torsion, for every $k\in \Z$ the following integer
\begin{equation}\label{polynomial}
\Hil(k)=\ind(\e{\,k\, \eta_0})
\end{equation}
does not depend on the choice of $\eta_0$. Moreover,
by \eqref{AT formula} we obtain that
\begin{equation}\label{HAT}
\Hil(k)= \Big( \sum_{h\geq 0} \frac{(k\,\eta_0)^h}{h!}\Big)\ttot[\M]= \sum_{h=0}^n k^h\left( \frac{\cc_1^h\,T_{n-h}}{\k0^h\,h!}\right)[\M]
\end{equation}
thus implying that, if $(\M,\J)$ has dimension $2n$, $\Hil(k)$ is a polynomial in $k$ of degree at most $n$. 
The polynomial $\Hil(z)$ defined as
\begin{equation}\label{Hilbert pol}
\Hil(z)= \sum_{h=0}^n a_h z^h=\sum_{h=0}^n \left( \frac{\cc_1^h\,T_{n-h}}{\k0^h\,h!}[\M]\right)z^h, \quad z\in \C 
\end{equation}
will be referred to as
the \emph{Hilbert polynomial of $(\M,\J)$}. 
Thus 
\begin{align}
& a_n=  \frac{1}{\k0^n\,n!} \cc_1^n[\M], \;\;\;\;\;\;a_{n-1}=\frac{1}{2\k0^{n-1}(n-1)!}\cc_1^n[\M],\nonumber \\
\label{ah} & a_{n-2}= \frac{1}{12\k0^{n-2}(n-2)!}(\cc_1^n+\cc_1^{n-2}\cc_2)[\M]\,, \;\;\;\;
\ldots\\
& a_0= T_n[\M] = \td(\M) \nonumber
\end{align}
The first properties of $\Hil(z)$ are given in the following 
\begin{prop}\label{properties P}
Let $\ac$ be an $S^1$-space with $N_0$ fixed points with zero negative weights.
Let $\cc_1$ be the first Chern class of the tangent bundle of $\M$ and assume that it is not torsion. Let $\k0\geq 1$ be the index of $(\M,\J)$,
$\Hil(z)$ the Hilbert polynomial, and $\deg(\Hil)$ its degree.
Then 
\begin{enumerate}
 \item\label{1a} $\Hil(0)=\td(\M)=N_0$;\\
 \item\label{3a} $\Hil(z)=(-1)^n \Hil(-\k0-z)\;\;\;$ for every $\;\;z\in \C$;\\
 \item\label{4a} $\deg(\Hil)\equiv n \mod 2$.
\end{enumerate}

\end{prop}

\begin{rmk}\label{H Ham}
By Lemma \ref{N0 1} and Proposition \ref{properties P} \eqref{1a}, note that if $(M,\omega)$ is a compact symplectic manifold supporting
a Hamiltonian $S^1$-action with isolated fixed points, then the Hilbert polynomial $\Hil(z)$ can never be identically zero.
\end{rmk}

\begin{rmk}\label{properties Ch}
Proposition \ref{properties P} \eqref{4a} implies that if there exists $k$ such that $a_{n-2h}=0$ for every $h=0,\ldots,k$, then
$a_{n-2h-1}=0$ for every $h=0,\ldots,k$.
\end{rmk}

\begin{proof}
Property \eqref{1a} follows from the definition of $\Hil(z)$ and Corollary \ref{todd genus comp}. 
By Lemma \ref{line admissible}, every line bundle $\LL$ such that $\pi(\cc_1(\LL))=k\,\pi(\eta_0)$ is admissible. So from
Proposition \ref{symmetries} we have that for all $k\in \Z$
$$
\Hil(k)=\ind(\e{\,k\, \eta_0})=(-1)^n \ind(\e{((-k-\k0) \eta_0)})=(-1)^n\Hil(-\k0-k)\,,
$$
and \eqref{3a} follows from observing that the polynomial given by $Q(z)=\Hil(z)-(-1)^n \Hil(-\k0-z)$ is zero for all $k\in \Z$, hence it must be identically zero.

In order to prove \eqref{4a} it is sufficient to notice that, if $\Hil(z)=\sum_{j=0}^ma_mz^m $, with $m=\deg(\Hil)$, from \eqref{3a} it follows that
$a_m=(-1)^{m+n}a_m$.
\end{proof}
Before proceeding with the main results of the section, we introduce some terminology that will be used in the discussion of the position of the roots of
$\Hil(z)$.
\begin{defin}\label{def: RVGo}
Fix a positive integer $k$.
\begin{itemize}
\item[1)] We denote by $\mathcal{T}_k$ the family of polynomials in $\R[z]$ that can be written as $C(z)\prod_{j=1}^{k-1}(z+j)$, where
  $C(z)\in \R[z]$ has all its roots on the line $l_{k}=\{x+\mathrm{i}y\in \C\mid x=-\frac{k}{2}\}$. 
 \item[2)] We define $\mathcal{S}_{k}$ to be the subset of the complex plane given by $$\mathcal{S}_{k}= \{x+\mathrm{i}y\in \C \mid -k  <x<0\}$$ 
 and we refer to it as the \emph{canonical strip} (centred at $-\frac{k}{2}$).
\item[3)] The subset of the complex plane $\mathcal{C}_{k}= \{z=x+\mathrm{i}y\in \C \mid y=0\;\;\mbox{or}\;\; x=-\frac{k}{2}\}$ is called the \emph{cross} at $-\frac{k}{2}$. 
\end{itemize}
\end{defin}

The terminology in 1) and 2) is inspired respectively by \cite{RV} and \cite{Go}. Indeed, in the beautiful note \cite{RV}, Rodriguez-Villegas analyses conditions which ensure a polynomial $H(z)\in \R[z]$ to belong to $\mathcal{T}_{k}$, for some $k\in \Z_{>0}$. 
In Sect.\;\ref{sec: generating fct} and Section \ref{sec: values k0} we explore connections 
among our results and those in \cite{RV}: we study under which conditions $\Hil(z)$ belongs to $\mathcal{T}_{\k0}$, for certain values of $\k0$.
In \cite{Go}, Golyshev analyses the position of the roots of the Hilbert polynomial of a Fano variety and a variety of general type. In particular, after adapting
his terminology to ours, he asks under which conditions all
the zeros of $\Hil(z)$ belong to the canonical strip $\mathcal{S}_{\k0}$.
In Section \ref{sec: values k0} we will study the position of the roots of $\Hil(z)$ in terms of inequalities in the Chern numbers and of $\k0$, when $\k0\geq n-2$ (see Remarks \ref{pos roots n+1}, \ref{pos roots n} and Corollaries \ref{pos roots n-1} and \ref{pos roots n-2}).

The next corollary is a straightforward consequence of Proposition \ref{properties P}. 
\begin{corollary}\label{property roots} 
Let $\ac$ be an $S^1$-space. Let $\k0\geq 1$ be the index of $(\M,\J)$, and
assume that the Hilbert polynomial $\Hil(z)$ is of positive degree $\deg(\Hil)>0$. If at least $\deg(\Hil)-3$ roots of $\Hil(z)$, counted with
multiplicity, belong to $\mathcal{C}_{\k0}$, then all the roots of $\Hil(z)$ belong to $\mathcal{C}_{\k0}$. In particular, if $n\leq 3$, then all the roots of $\Hil(z)$ belong to $\mathcal{C}_{\k0}$.
\end{corollary}
\begin{proof}
Let $h$ be the number of roots, counted with multiplicity, which belong to $\mathcal{C}_{\k0}$; by assumption $h\geq \deg(\Hil)-3$. 
Suppose that one of the remaining $\deg(\Hil)-h$ roots, $z_0\in \C$, does not belong to $\mathcal{C}_{\k0}$. Then, by Proposition \ref{properties P} \eqref{3a}, 
we have that $z_1=-\k0-z_0$ is also a root, and since $\Hil(z)\in \R[z]$, the complex conjugates $z_2=\overline{z_0}$ and $z_3=-\k0-\overline{z_0}$ are also roots.
Since $z_0\notin \mathcal{C}_{\k0}$, it follows that $z_i\neq z_j$ for $i\neq j$, and $z_i\notin \mathcal{C}_{\k0}$ for $i=0,1,2,3$, implying that $\Hil(z)$ has at least $h+4\geq \deg(\Hil)+1$ roots, which is impossible
since we are assuming $\Hil(z)$ to be non identically zero.

\end{proof}

We are now ready to prove Theorem \ref{main theorem}.
\begin{proof}[Proof of Theorem \ref{main theorem}]
Choose $\eta_0$ and $\tau$ in $H^2(\M;\Z)$ such that $\cc_1=\k0 \eta_0 + \tau$,
where $\tau\in \tor(H^2(\M;\Z))$. By Lemma \ref{line admissible}, both $\eta_0$ and $\tau$ admit equivariant extensions
$\eta_0^{S^1}$ and $\tau^{S^1}$ in $H_{S^1}^2(\M;\Z)$. Since $\tau^{S^1}(p)$ does not depend on $p\in \M^{S^1}$ (see \eqref{trivial constant}),
it follows that $\ce(p)=\k0 \eta_0^{S^1}(p)+c$ for all $p\in \M^{S^1}$, for some $c\in \Z$.
Thus by Proposition \ref{eq index zero} we have that
\begin{equation}\label{key equation}
\ind_{S^1}(\e{\,k\eta_0^{S^1}})=0 \quad \mbox{for all}\quad k=-1,-2,\ldots,-\k0+1\,,
\end{equation}
and by combining \eqref{K commutes} and \eqref{key equation} we have that
$$
\Hil(k)=\ind(\e{\,k\eta_0})=r_K(\ind_{S^1}(\e{\,k\eta_0^{S^1}}))=0 \quad \mbox{for all}\quad k=-1,-2,\ldots,-\k0+1\,,
$$
and \eqref{H=0 even} follows.

In order to prove \eqref{bound k0}, observe that by \eqref{H=0 even} the set of roots of $\Hil(z)$ contains $C_0=\{-1,-2,\ldots,-\k0+1\}$, thus if $\Hil(z)\not\equiv 0$ we must have that $|C_0|=\k0-1\leq \deg(\Hil)\leq n$.
\end{proof}
Note that by Proposition \ref{properties P}, $\Hil(z)$
has a different behaviour depending on whether $N_0=0$ or not. 
\begin{corollary}\label{bound on k0}
Let $\ac$ be an $S^1$-space, and assume that $\cc_1$ is not torsion. Let $\k0\geq 1$ be the index of $\ac$ and
$\Hil(z)$ the Hilbert polynomial. Let
$N_0$ be the number of fixed points with $0$ negative weights.
Then:
\begin{itemize}
 \item[({\bf i})] If $N_0\neq 0\;\;\;$ then  $\;\;\;1\leq \k0\leq \deg(\Hil)+1\leq n+1$;
 \item[({\bf ii})] If $N_0=0\;\;\;$ then either $\;\;\deg(\Hil)>0$ and $1\leq \k0\leq \deg(\Hil)-1\leq n-1$, or
 $\Hil(z)\equiv 0
 $, the latter being equivalent to $\cc_1^h\, T_{n-h}[\M]=0 \quad \mbox{for every}\quad h=0,\ldots,n\,$.
\end{itemize}
\end{corollary}
\begin{proof}
If $\k0=1$, the inequalities in ({\bf i}) clearly hold. Assume $\k0\geq 2$.
Observe that if $N_0\neq 0$ then Proposition \ref{properties P} \eqref{1a} implies that $\Hil(z)$ is not identically zero, and ({\bf i}) follows from \eqref{bound k0}. 

Suppose that $N_0=0$. Observe that in this case we must have\footnote{Indeed, in Section \ref{examples} it will be proved that $N_0=0$ implies $n\geq 3$, see Prop.\ \ref{dim 4}.} $n\geq 2$. Indeed, for $n=1$ it is impossible to have $N_0=0$, since by \eqref{NiN} 
we would have $N_0=N_1=0$, and hence $|\M^{S^1}|=0$. 
By Proposition \ref{properties P} \eqref{1a} and \eqref{3a}, and by \eqref{H=0 even}, we have that 
the set of roots of 
$\Hil(z)$ contains $C'_0=\{0,-1,\ldots,-\k0\}$. It follows that, if $\Hil(z)$ is not identically zero, then $|C'_0|=\k0+1\leq \deg(\Hil)\leq n$.
\end{proof}
A consequence of Corollary \ref{bound on k0} in the symplectic category is the following:
\begin{corollary}\label{bound on k0 s}
Let $(\M,\omega)$ be a compact, connected symplectic manifold, and $\k0$ the associated index. Then:
\begin{itemize}
\item[({\bf i'})] If $(\M,\omega)$ can be endowed with a Hamiltonian $S^1$-action with isolated fixed points, then $\;\;\;1\leq \k0\leq \deg(\Hil)+1\leq n+1$;
\item[({\bf ii'})] If $(\M,\omega)$ can be endowed with a non-Hamiltonian $S^1$-action with isolated fixed points, then there are three possibilities:
\begin{itemize}
\item[(a)] $\k0=0$, i.e.\ $\cc_1$ is torsion;
\item[(b)] $\k0>0$ and $\Hil\equiv 0$, the latter being equivalent to $\cc_1^h\, T_{n-h}[\M]=0 \quad \mbox{for every}\quad h=0,\ldots,n\,$;
\item[(c)] $\k0>0$, $\deg(\Hil)>0$ and $1\leq \k0\leq \deg(\Hil)-1\leq n-1$.
\end{itemize}
\end{itemize}
\end{corollary}
\begin{proof}
In order to prove ({\bf i'}) it is sufficient to notice that, by Lemma \ref{Lemma:c1 not torsion}, we must have $\k0>0$. Then the claim follows from Lemma \ref{N0 1} and Corollary \ref{bound on k0} ({\bf i}).
The only non trivial thing to prove in ({\bf ii'}) is the upper bound on the index in (c). But this follows by combining Lemma \ref{N0 1} and Corollary \ref{bound on k0} ({\bf ii}).
\end{proof}
Corollary \ref{minimal chern ham} follows from Corollary \ref{bound on k0} ({\bf i'})  and the discussion in Remark \ref{mcn}.
\begin{rmk}\label{rmk 1}
(a') In the $6$-dimensional example $(\widetilde{M},\omega)$ constructed by Tolman \cite{T3}, the image of $\cc_1^{S^1}(\widetilde{\M})$ under the restriction map $i^*\colon H^2_{S^1}(\widetilde{M};\Z)\to H^2_{S^1}(\widetilde{M}^{S^1};\Z)$
is identically zero. Such restriction is zero when, for instance, $\cc_1$ is torsion in $H^2(\M;\Z)$ (see \cite[Lemma 4.1]{GPS}). However, to the best of the author's knowledge,
it is still not known whether
$\cc_1(\widetilde{M})$ is torsion. 
\\
(b') Note that, under the hypothesis of ({\bf ii'}), if $\k0\geq n$ then $\Hil\equiv 0$.
\end{rmk}
\begin{rmk}[{\bf Comparison with Hattori's results}]\label{hattori rmk}
In \cite{Ha} Hattori analyses inequalities which are similar to those in Corollary \ref{bound on k0}, provided that $\ac$ is an $S^1$-space endowed with a suitable quasi-ample line bundle, defined as follows.
An equivariant line bundle $\LL^{S^1}$ is \emph{fine} if the restrictions of $\LL^{S^1}$ at the fixed points are mutually distinct
$S^1$-modules, i.e.\; if $\LL^{S^1}(p_i)= t^{a_i}\neq t^{a_j}=\LL^{S^1}(p_j)$
for every $p_i\neq p_j $ in $\M^{S^1}$.
It is \emph{quasi-ample} if it is fine and its first (non equivariant) Chern class satisfies $ \cc_1(\LL^{S^1})^n[\M]\neq 0$.
In \cite[Theorem 5.1]{Ha} the author proves that if $\ac$ possesses a quasi ample line bundle $\LL^{S^1}$,
and its first (non equivariant) Chern class satisfies $\cc_1=k\, \cc_1(\LL^{S^1})$ for some
$k\in \Z_{>0}$, then $k\leq n+1 \leq \chi(\M)$. 
Thus, if the equivariant line bundle $\eta_0^{S^1}$ defined in the proof of Theorem \ref{main theorem} is quasi-ample, Hattori's results
imply that $\k0\leq n+1\leq \chi(\M)$. 
Observe that in Corollary \ref{bound on k0} ({\bf i}), we do not require the existence of a quasi-ample line bundle; we assume instead $N_0\neq 0$.
We also remark that if $\cc_1^n[\M]=0$ then $\eta_0^{S^1}$ cannot be quasi-ample; on the other hand, if $ \cc_1^n[\M]=0$ and $N_0\neq 0$, Corollary \ref{bound on k0} ({\bf i}) gives a better upper bound on $\k0$, since the vanishing of $ \cc_1^n[\M]$ implies that $\deg(\Hil)\leq n-2$, thus giving $\k0\leq n-1$ (see Remark \ref{c1n=0}).
\end{rmk}

\begin{rmk}\label{c1n=0}
From \eqref{ah}, Proposition \ref{properties P} \eqref{4a} and Corollary \ref{bound on k0} it follows that if $\k0\geq 1$:
\begin{itemize}
\item If $ \cc_1^n[\M]=0$ and $N_0\neq 0$, then $\deg(\Hil(z))\leq n-2$ and $ \k0\leq n-1$;\\
\item If $\cc_1^n[\M]=0$ and $N_0=0$, then $\k0\leq n-3$ or $\Hil(z)\equiv 0$.
\end{itemize}
Similarly,
\begin{itemize}
\item If $\cc_1^n[\M]=\cc_1^{n-2}\cc_2[\M]=0$ and $N_0\neq 0$, then $\deg(\Hil(z))\leq n-4$ and $ \k0\leq n-3$;\\
\item If $\cc_1^n[\M]=\cc_1^{n-2}\cc_2[\M]=0$ and $N_0=0$, then $\k0\leq n-5$ or $\Hil(z)\equiv 0$.
\end{itemize}
\end{rmk}

\begin{rmk}\label{nec non Ham}
Observe that by Corollary \ref{bound on k0 s} ({\bf ii'}) and \eqref{ah} it follows that if $(\M,\omega)$ supports a non-Hamiltonian action and $\k0\geq n$, then
$\cc_1^n[\M]=\cc_1^{n-2}\cc_2[\M]=0$. In Theorem \ref{nHam-char} we strengthen this fact and
 prove that for $(\M,\omega,S^1)$ with $\k0\geq n$, the vanishing of one of these Chern numbers is indeed equivalent to having a non-Hamiltonian
action. Moreover, if $\k0=n-2$ or $\k0=n-1$, then a suitable linear combination of those Chern number is zero if and only if the action is non-Hamiltonian. 
\end{rmk} 
As we have already observed before (Lemma \ref{Lemma:c1 not torsion}), a compact symplectic manifold with $\cc_1$ torsion cannot support any Hamiltonian circle action.
If $\cc_1$ is not torsion, a criterion to conclude the same is given by the following
%%%%%%%
\begin{corollary}\label{cor non ham 2}
Let $(\M,\omega)$ be a compact, connected symplectic manifold of dimension $2n$ with index $\k0>0$.
If 
$$
\cc_1^h\,T_{n-h}[\M]=0 \quad \mbox{for all}\quad h\geq 2\k0-n+2\Big\lfloor\frac{n-\k0}{2}\Big\rfloor
$$
then the manifold cannot support any Hamiltonian circle action with isolated fixed points. 
\end{corollary}
\begin{proof}
First of all observe that $\;2\k0-n+2\Big\lfloor\frac{n-\k0}{2}\Big\rfloor\geq \k0-1$, and equality holds if and only if $n\not\equiv \k0\mod{2}$.
By definition of Hilbert polynomial (see \eqref{ah}),
having
$\cc_1^h\, T_{n-h}[\M]=0$ for all $h\geq0$ implies that $\Hil\equiv 0$. If $\k0\geq 2$, having  $\cc_1^h\, T_{n-h}[\M]=0$ for all $h\geq \k0-1$ implies
that $\deg(\Hil)\leq \k0-2$.
However, as a consequence of Theorem \ref{main theorem}, $\Hil(z)$ has at least $\k0-1$ zeroes, so $\cc_1^h\, T_{n-h}[\M]=0$ for all $h\geq\k0-1$ implies that
$\Hil\equiv 0$. 

By Remark \ref{H Ham}, the Hilbert polynomial 
of a symplectic manifold with a Hamiltonian $S^1$-action and isolated fixed points can never be identically zero, and the corollary follows from the discussion above for $n\not\equiv \k0\mod{2}$.
If $n\equiv \k0\mod{2}$ then, by Proposition \ref{properties P} \eqref{4a} we have that $\deg(\Hil)\leq \k0-1$ implies $\deg(\Hil)\leq \k0-2$, and the conclusion holds in this case too.
\end{proof}
%%%%%

Another consequence of Theorem \ref{main theorem} is the following
\begin{corollary}\label{extra root -k02}
Let $\ac$ be an $S^1$-space, and assume its index $\k0$ is non-zero. Let $\Hil(z)$ be the Hilbert polynomial. 
\begin{equation}\label{extra root}
 \mbox{If} \quad n\equiv\k0\mod 2\quad \mbox{then} \;\;\;\Hil\Big(-\frac{\k0}{2}\Big)=0\,.
 \end{equation}
Moreover, if $\Hil(z)\not\equiv 0$ and $n\equiv \k0\equiv 0\mod 2$, then the multiplicity of the root $-\frac{\k0}{2}$ is at least $2$.
\end{corollary}
\begin{proof}
Observe that, if $\k0\geq 2$, by \eqref{H=0 even} we have that $\widetilde{\Hil}(z)=\displaystyle\frac{\Hil(z)}{\prod_{j=1}^{\k0-1}(z+j)}$
is a polynomial. The same conclusion follows if $\k0=1$ by setting the empty product to be $1$. Hence
 by Proposition \ref{properties P} \eqref{3a} we have that for all $\k0\geq 1$
\begin{equation}\label{H0 tilde}
\widetilde{\Hil}(-\k0-z)=\frac{\Hil(-\k0-z)}{\prod_{j=1}^{\k0-1}(-\k0-z+j)}=\frac{(-1)^n\Hil(z)}{(-1)^{\k0-1}\prod_{j=1}^{\k0-1}(z+j)}=(-1)^{n-\k0+1}\widetilde{\Hil}(z)\;.
\end{equation}
Hence if $n\equiv \k0\mod 2$, from \eqref{H0 tilde} it follows that $\widetilde{\Hil}(-\frac{\k0}{2})=0$, thus proving \eqref{extra root}.
Finally, if $\k0$ is even, then $-\frac{\k0}{2}\in \{-1,\ldots,-\k0+1\}\subset \Z$, hence it is a root of both $\prod_{j=1}^{\k0-1}(z+j)$ and $\widetilde{\Hil}(z)$.
\end{proof}

From Theorem \ref{main theorem} we also have the following refinement of Corollary \ref{property roots}, which concerns the position of the roots of $\Hil(z)$.
\begin{corollary}\label{position roots}
Let $\ac$, $\k0$ and $\Hil(z)$ be as in Theorem \ref{main theorem}, and assume that $\deg(\Hil)>0$. If $\k0\geq n-2$ then all the roots of $\Hil(z)$ belong to $\mathcal{C}_{\k0}$. 
\end{corollary}

The next corollary gives useful equations in the Chern numbers
depending on the index $\k0$ and the parity of $n-\k0$.
\begin{corollary}[{\bf Equations in the Chern numbers}]\label{cor equations chern numbers}
Let $\ac$ be as in Theorem \ref{main theorem}. Then 
\begin{equation}\label{equations chern numbers}
\sum_{h=0}^n \frac{1}{h!}\left( \frac{k}{\k0}\right)^h \cc_1^h\, T_{n-h}[\M]=0\quad \;\;\;\;\mbox{for all}\;\; k \in \{-1,-2,\ldots,-\k0+1\}\,.
\end{equation}
Moreover, if $n\equiv \k0\mod 2$ then
\begin{equation}\label{k02}
\sum_{h=0}^n \frac{(-1)^h}{2^h h!}\cc_1^h\, T_{n-h}[\M]=0\;,
\end{equation}
and if $n\equiv \k0 \equiv 0\mod 2$ then
\begin{equation}\label{k02 2}
\sum_{h=1}^n \frac{(-1)^{h-1}}{2^{h-1} (h-1)!}\cc_1^h\, T_{n-h}[\M]=0
\end{equation}
\end{corollary}
\begin{proof}
It is sufficient to notice that \eqref{k02 2} is equivalent to having $\Hil'(-\frac{\k0}{2})=0$, and the proof of Corollary \ref{cor equations chern numbers} is a direct consequence of Theorem \ref{main theorem}, Corollary \ref{extra root -k02} and
the definition of Hilbert polynomial
\eqref{Hilbert pol}.
\end{proof}
Thus the cases in which we can derive more restrictions on the Chern numbers are when $\k0$ is ``large'' (see Section \ref{sec: values k0}). 
\\$\;$

Before proceeding with the analysis of $\Hil(z)$ for different values of $\k0$, in the next subsection we study the properties of
the generating function of the sequence $\{\Hil(k)\}_{k\in \mathbb{N}}$.

\subsection{The generating function associated to the Hilbert polynomial}\label{sec: generating fct}

We recall that the \emph{generating function} of a sequence $\{b_k\}_{k\in \mathbb{N}}\subset \R$ is the formal power series
$$
P(t)=\sum_{k\geq 0} b_k t^k\,.
$$

The following result is due to Popoviciu \cite{Po} (see also 
\cite[Corollary 4.7]{St}).
\begin{prop}[Popoviciu]\label{Pop}
Let $H(z)$ be a polynomial of degree $m$ and $P(t)$ the generating function of the sequence $\{H(k)\}_{k\in \mathbb{N}}$. Then 
\begin{equation}\label{sym P}
P(t^{-1})=(-1)^{m+1}t^{k_0}P(t)  
 \end{equation}
  for some $k_0\in \Z$ if and only if $k_0\geq 1$,  
 \begin{equation}\label{zeros H}
  H(-1)=H(-2)=\cdots = H(-k_0+1)=0
 \end{equation}
 and 
 \begin{equation}\label{symmetry}
  H(k)=(-1)^m H(-k_0-k)\quad \mbox{for every}\quad k\in \Z\,.
 \end{equation}
\end{prop}

As a consequence of the properties satisfied by $\Hil(z)$, we have the following

\begin{prop}\label{gen fct hilbert}
Let $\ac$ be an $S^1$-space and assume its index is non-zero. Let $\Hil(z)$ be the associated Hilbert polynomial of degree $\deg(\Hil)=m$. Let $N_0$
be the number of fixed points with $0$ negative weights.
Then the generating function $\Gen(t)$ of $\{\Hil(k)\}_{k\in \mathbb{N}}$  
is given by
\begin{equation}\label{gener fct}
\Gen(t)=\frac{\U(t)}{(1-t)^{m+1}} 
\end{equation}
where $\U(t)$ is a polynomial in $\R[t]$ such that $\U(0)=N_0$, with 
\begin{equation}\label{prop Gen1}
\Gen(t^{-1})=(-1)^{m+1}t^{\k0} \Gen(t)
\end{equation}
and 
\begin{equation}\label{prop U}
 \U(t^{-1})=t^{\k0-m-1}\U(t)\,.
\end{equation}
Moreover, if $\Hil(z)\not\equiv 0$, then
\begin{equation}\label{degree U}
\frac{m+1-\k0}{2}\leq \deg(\U)\leq m+1-\k0\,,
\end{equation}
and $\deg(\U)=m+1-\k0$ if and only if $N_0\neq 0$.
Here $\deg(\U)$ denotes the degree of $\U$.
\end{prop}
Thus, by Lemma \ref{N0 1}, if $(\M,\omega)$ is a compact symplectic manifold and the $S^1$-action is Hamiltonian, then 
the polynomial $\U(t)$ is of degree $m+1-\k0$.
\begin{proof}
It is well known that the generating function of a sequence $\{H(k)\}_{k\in \mathbb{N}}$, where $H\in \R[z]$ is a polynomial of degree $m$,
is of the form given by \eqref{gener fct}, where $\U(t)\in \R[t]$ is a polynomial of degree at most equal to $m$. In order to prove that $\U(0)=N_0$, observe that 
$$\Gen(t)=\frac{\U(t)}{(1-t)^{m+1}}= \U(t)\sum_{k\geq 0}\binom{m+k}{m}t^k=\U(0)+tQ(t)$$
for some formal power series $Q(t)\in \R[[t]]$. Thus $\U(0)=\Gen(0)=\Hil(0)$, and by Proposition \ref{properties P} \eqref{1a} $\Hil(0)=N_0$.

As for \eqref{prop U}, observe that by Theorem \ref{main theorem} \eqref{H=0 even}, if $\k0\geq 2$ we have that \eqref{zeros H} is satisfied for $k_0=\k0$, the index of $(\M,\J)$.
If $\k0=1$ \eqref{zeros H} is trivially satisfied, since it is the empty condition. 
Moreover, by Proposition \ref{properties P} \eqref{3a} and \eqref{4a}, we have that \eqref{symmetry} is satisfied as well. Thus by Proposition \ref{Pop} 
the generating function $\Gen(t)$ of $\{\Hil(k)\}_{k\in \mathbb{N}}$ satisfies \eqref{prop Gen1}, obtaining
$$
(-1)^{m+1}\frac{t^{m+1}\U(t^{-1})}{(1-t)^{m+1}}=\Gen(t^{-1})=(-1)^{m+1}t^{\k0}\Gen(t)=(-1)^{m+1}\frac{t^{\k0}\U(t)}{(1-t)^{m+1}}
$$
and \eqref{prop U} follows.

Let $e=\deg(\U)$ and $\U(t)=\alpha_0+\alpha_1t+\cdots +\alpha_e t^e$. By \eqref{prop U} we have that
\begin{equation}\label{coeff U}
\alpha_et^{m+1-\k0-e}+\alpha_{e-1}t^{m+1-\k0-e+1}+\cdots + \alpha_0t^{m+1-\k0}=\alpha_0+\alpha_1t+\cdots +\alpha_e t^e\,,
\end{equation}
hence we must have $0\leq m+1-\k0-e\leq e$, and \eqref{degree U} follows.
The equality in \eqref{coeff U} also implies that $\alpha_0=\U(0)=N_0\neq 0$ if and only if $m+1-\k0-e=0$.
\end{proof}

We recall that a polynomial of degree $e$, $U(t)=\alpha_0+\alpha_1t+\cdots + \alpha_e t^e$, is called \emph{self-reciprocal} if 
\begin{equation}\label{self-rec}
t^{e}U(t^{-1})=U(t)\,.
\end{equation}
Such a polynomial is sometimes also referred to as a \emph{palindromic}, since \eqref{self-rec} is equivalent to saying that
the list of coefficients $\alpha_0\,\alpha_1\,\cdots \alpha_e$ is a palindrome, i.e.\;$\alpha_i=\alpha_{e-i}$ for every $i$.
\begin{corollary}\label{U palindrom}
With the same notation of Proposition \ref{gen fct hilbert}, we have that:
\begin{itemize}
\item[({\bf i})] $\U(t)$ is divisible by $t^{m+1-\k0-e}$, where
$e=\deg(\U)$, and the polynomial $t^{e+\k0-m-1}\U(t)$ is self-reciprocal.  
\item[({\bf ii})] If $(\M,\omega)$ is a symplectic manifold and the $S^1$-action is Hamiltonian,
then $\U(t)$ is self-reciprocal. Moreover if $(\M,\omega)$ is monotone with $\cc_1=\k0 [\omega]$, then 
$\deg(\U)=n+1-\k0$.
\end{itemize}
\end{corollary}
\begin{proof}
The claims in ({\bf i}) are a consequence of \eqref{prop U} and \eqref{coeff U}. 
If $\deg(\U)=e=m+1-\k0$, which by Proposition \ref{gen fct hilbert} is equivalent to having $N_0\neq 0$,
we obtain that $\U(t)$ is self-reciprocal, and the first claim in ({\bf ii}) follows from Lemma \ref{N0 1}.
The second claim follows from observing that monotonicity implies $ \cc_1^n[\M] \neq 0$, hence $\deg(\Hil)=n$.
\end{proof}
\begin{rmk}\label{num of cds}
Observe that the polynomial $\U(t)$ determines $\Gen(t)$ which, in turns, determines $\Hil(z)$. Thus the Hilbert polynomial, and hence
all the combinations of Chern numbers $\cc_1^h T_{n-h}[\M]$, for $h=0,\ldots,n$, are completely determined by the coefficients of $\U(t)$.
Moreover, if $N_0$ is given, the coefficient of degree zero in $\U(t)$ is known, since by Proposition \ref{gen fct hilbert} $\U(0)=N_0$.
In conclusion, from Corollary \ref{U palindrom} it follows that the number of coefficients of $\U(t)$ to determine is at most
equal to $\floor*{\displaystyle\frac{m-\k0-1}{2}}+1$. This explains why
the number of conditions that completely determine the Hilbert polynomial (and hence the combinations of Chern numbers $\cc_1^h T_{n-h}[\M]$)
is the same when $\k0=n+1-2k$ and $\k0=n-2k$, for every $k\in \Z$ such that $0\leq k\leq \frac{n-1}{2}$.
\end{rmk}
In the beautiful note \cite{RV}, the author analysis the position of the roots of $\Hil(z)$ in terms of those of $\U(t)$,
 deriving the following
\begin{thm}[Rodriguez-Villegas \cite{RV}]\label{RV theorem}
Let the notation be as in Proposition \ref{gen fct hilbert}, and $\mathcal{T}_k$ as in Definition \ref{def: RVGo}. 
Assume that $\Hil(z)\not\equiv 0$ and that all the roots of $\U(t)$ are on the unit circle. Then $\Hil(z)$ belongs to $\mathcal{T}_{\k0}$.
\end{thm}
In the next section we analyse the different expressions of $\U(t)$ for $\k0\in \{n-2,n-1,n,n+1\}$.
As a consequence, we prove that if $\k0=n$ or $\k0=n+1$, then $\Hil(z)$ always belongs to $\mathcal{T}_{\k0}$ (unless $\Hil(z)\equiv 0$).
If $\k0=n-2$ or $n-1$, we
derive necessary and sufficient conditions on the Chern numbers that 
ensure $\Hil(z)$ to be in $\mathcal{T}_{\k0}$, or more in general that ensure its roots to be on the canonical strip $\mathcal{S}_{\k0}$ (see Corollaries \ref{pos roots n-1} and \ref{pos roots n-2}).  As a byproduct, we prove that when $N_0=1$ and $n$ is big enough, then $\Hil(z)$ belongs to $\mathcal{T}_{\k0}$ \emph{if and only if}
the roots of $\U(t)$ are on the unit circle (see Corollaries \ref{RV1} and \ref{RV2}).

\subsection{Connection with Ehrhart polynomials}\label{connections ehrhart}
Some of the results in Section \ref{equations chern} can be regarded as a generalisation of what is already known for the Ehrhart polynomial of a reflexive polytope.
The link between Hilbert polynomials of $S^1$-spaces and Ehrhart polynomials of reflexive polytopes is given by monotone symplectic toric manifolds.

Suppose that $(\M,\omega)$ is a compact symplectic manifold of dimension $2n$, and that the $S^1$-action extends to a toric action, i.e.\
$S^1$ is a circle subgroup in an $n$-dimensional torus $\mathbb{T}^n$ which is acting effectively on $(\M,\omega)$ 
 with moment map $\Psi\colon (\M,\omega) \to Lie(\mathbb{T}^n)^*$. We identify $Lie(\mathbb{T}^n)^*$ with $\R^n$, and let the dual lattice of $\T^n$ be $\Z^n$.  
 By the Atiyah \cite{At1} and Guillemin-Sternberg \cite{GS82} convexity theorem, we know
  that $\Psi(\M)=: \Delta$ is a convex polytope, more precisely it is the convex hull 
 of its vertices, which coincide with the images of the fixed points of the $\T^n$ action. 
  Suppose that $(\M,\omega)$ is also \emph{monotone} and rescale the symplectic form so that $\cc_1=\k0 [\omega]$ (so $[\omega]$
 is primitive in $H^2(\M;\Z)$, which is torsion free in this case). Choose the moment map $\Psi$ so that all the vertices of $\Delta$ belong to the lattice $\Z^n$: we call such polytope $\Delta$ \emph{primitive} and \emph{integral}. 
 As a consequence of a result of Danilov \cite{Da}, we have that 
 \emph{the Hilbert polynomial $\Hil(z)$ of $(\M,\omega)$ coincides with the Ehrhart polynomial $i_{\Delta}(z)$ of $\Delta$}.
 Moreover, it is well-known that there exists a (unique) $k\in \Z_{>0}$ such that the dilated polytope 
  $\Delta'=k \Delta$, suitably translated by an integer vector, is \emph{reflexive}\footnote{An integral polytope $\mathcal{P}\subset \R^n$ of dimension $n$ is reflexive if it contains the origin in its interior, and its dual polytope $\mathcal{P}^*=\{\mathbf{x}\in \R^n\mid \mathbf{x}\cdot \mathbf{y}\geq -1\mbox{ for all }\mathbf{y}\in \mathcal{P}\}$
 is also integral.}. By a result of Hibi \cite{Hibi}, this is equivalent to saying that the Ehrhart polynomial $i_{\Delta'}(z)$ and its associated generating function $P_\Delta'(t)=\frac{U(t)}{(1-t)^{n+1}}$
 satisfy 
 \begin{equation}\label{deltas}
 i_{\Delta'}(z)=(-1)^n i_{\Delta'}(-1-z)\quad\quad\mbox{and}\quad\quad P_{\Delta'}(t^{-1})=(-1)^{n+1}t \,P_{\Delta'}(t).
 \end{equation}
The following gives a combinatorial characterisation of the index $\k0$ of $(\M,\omega)$ (which, by Remark \ref{mcn}, coincides with the minimal Chern number): 
\begin{lemma}\label{k0 equivalence}
Let $(\M,\omega,\T,\Psi)$ be a monotone symplectic toric manifold, with symplectic form satisfying $\cc_1=\k0[\omega]$. Consider the primitive integral moment polytope image $\Delta$.
Then the index $\k0$ is the unique integer so that $\Delta'=\k0 \Delta$ is reflexive.
\end{lemma}
\begin{proof}
First of all observe that, from $\Delta'=k\Delta$ we have $i_{\Delta}(z)=i_{\Delta'}\big(\frac{z}{k}\big)$ for every $z\in \C$.
Moreover, as mentioned before, $\Hil(z)=i_\Delta(z)$.
 So from \eqref{deltas} we have that 
$$
\Hil(z)=i_{\Delta}(z)=i_{\Delta'}\Big(\frac{z}{k}\Big)=(-1)^n i_{\Delta'}\Big(-1-\frac{z}{k}\Big)=(-1)^n i_\Delta(-k-z)=(-1)^n \Hil(-k-z)\,,
$$
for every $z\in \C$. By Remark \ref{H Ham}, $\Hil(z)$ is a nonzero polynomial, so Proposition \ref{properties P} \eqref{3a} implies that $\k0=k$. 
\end{proof}
It is in this sense that we can regard the symmetry property of $\Hil(z)$ (i.e.\ Proposition \ref{properties P} \eqref{3a}) and the results in Proposition \ref{gen fct hilbert} as a generalisation of 
\eqref{deltas}.
%%%%%%%%%%

\section{Computation of $\Hil(z)$ and Chern numbers for some values of $\k0$}\label{sec: values k0}
In this section, we compute explicitly the Hilbert polynomial $\Hil(z)$ and its associated generating 
function for $\k0\geq n-2$ and $\k0\neq 0$, deriving more properties of the Chern numbers of $\ac$. 

Let $\sigma_j(x_1,\ldots,x_n)$ be the $j$-th elementary symmetric polynomials 
in $x_1\ldots,x_n$, for $j=0,\ldots,n$, and let $\left[ \begin{array}{c} n \\ k \end{array} \right]$ be the \emph{unsigned Stirling numbers of the
first kind}, where $k,n \in \mathbb{N}$ and $1\leq k\leq n$, satisfying
\begin{equation}\label{stirling}
(x)^{(n)}=x(x+1)\cdots (x+n-1)=\sum_{k=0}^n \left[ \begin{array}{c} n \\ k \end{array} \right]x^k\,,
\end{equation}
where $(x)^{(n)}$ is the rising factorial. 
Thus we have the relation:
\begin{equation}\label{stirling permutation}
\sigma_k(1,2,\ldots,n)=\left[ \begin{array}{c} n+1 \\  \\ n-k+1 \end{array} \right]\,,
\end{equation}
and the following well-known identities:
\begin{align}
& \sigma_0(1,2,\ldots,n)=\left[ \begin{array}{c} n+1 \\ n+1 \end{array} \right]=1\nonumber\\
\label{sigma1}& \sigma_1(1,2,\ldots,n)=\left[ \begin{array}{c} n+1 \\ n \end{array} \right] = \binom{n+1}{2}\\
\label{sigma2} & \sigma_2(1,2,\ldots,n)= \left[ \begin{array}{c} n+1 \\ n-1 \end{array} \right] =\frac{1}{4}(3n+2)\binom{n+1}{3}=\frac{(3n+2)(n+1)n(n-1)}{24}
\end{align}
Observe that by Corollary \ref{bound on k0}, if $\k0>n+1$ then $\Hil(z)\equiv 0$ and $ \cc_1^h T_{n-h}[\M]=0$ for every $h=0,\ldots,n$. So in the rest of the section
we will focus on the cases in which $0<\k0\leq n+1$.

Before beginning, we remind the reader that the Hilbert polynomial of $\C P^n$ is given by $\frac{\prod_{j=1}^n(z+j)}{n!}$.
\begin{prop}[$\k0=\mathbf{n+1}$]\label{cor n+1}

Let $\ac$ be and $S^1$-space with index $\k0=n+1$. Let $N_0$ be the number of fixed points with $0$ negative weights.
 Then 
\begin{equation}\label{H k0=n+1}
\Hi(z)= \frac{N_0}{n!}\prod_{j=1}^n(z+j)=N_0 \Hil_{\C P^n}(z)\,,
\end{equation}
where $\Hil_{\C P^n}(z)$ is the Hilbert polynomial of $\C P^n$, and for every $h=0,\ldots,n$ we have 
\begin{equation}\label{n+1 precise}
\cc_1^h\, T_{n-h}[\M]=N_0\frac{h!(n+1)^h}{n!}\left[ \begin{array}{c} n+1 \\ h+1 \end{array} \right]=N_0\, \cc_1^h\,T_{n-h}[\C P^n].
\end{equation}
In particular 
\begin{equation}\label{c1 n+1}
\cc_1^n[\M]=N_0 (n+1)^n\, 
\end{equation}
and
\begin{equation}\label{c1c2}
 \cc_1^{n-2}\cc_2[\M]=N_0\frac{n(n+1)^{n-1}}{2}\,.
\end{equation}
Moreover, the generating function of $\{\Hil(k)\}_{k\in \mathbb{N}}$ is given by
\begin{equation}\label{gen fct n+1}
\Gen(t)=N_0\frac{1}{(1-t)^{n+1}}
\end{equation}
\end{prop}
%%%%%%%%%%%%%

\begin{rmk}\label{pos roots n+1}
From \eqref{gen fct n+1} and Proposition \ref{gen fct hilbert} we have that in this case $\U(t)=N_0$, and
if $N_0\neq 0$, the zeros of $\Hi(z)$ coincide with the integers greater than $-\k0=-(n+1)$ and smaller than $0$,
thus in particular $\Hil(z)$ belongs to $\mathcal{T}_{n+1}$, and hence all its roots are on the canonical strip $\mathcal{S}_{n+1}$ (see Theorem \ref{RV theorem}).
\end{rmk}

\begin{proof}[Proof of Proposition \ref{cor n+1}]
If $N_0=0$ then all the claims in Proposition \ref{cor n+1} follow from Corollary \ref{bound on k0} ({\bf ii}). 
Suppose that $N_0\neq 0$. By Proposition \ref{properties P} \eqref{1a}, $\Hi(z)$ is a nonzero polynomial which, by Theorem \ref{main theorem}
\eqref{H=0 even}, has roots $-1,-2,\ldots,-n$ (note that in this case $\k0\geq 2$). Thus $\Hi(z)=\alpha \prod_{j=1}^n(z+j)$.
In order to find $\alpha$ we can use Proposition \ref{properties P} \eqref{1a}, obtaining $\Hi(0)=\alpha\, n!=N_0$, and \eqref{H k0=n+1} follows.
For $h=0,\ldots,n$, the term of degree $h$ on the right hand side of \eqref{H k0=n+1} is given by $\frac{N_0}{n!}\sigma_{n-h}(1,2,\ldots,n)=\frac{N_0}{n!}\left[ \begin{array}{c} n+1 \\ h+1 \end{array} \right]$.
On the other hand, the term of degree $h$ on the left hand side of \eqref{H k0=n+1} can by computed by using \eqref{Hilbert pol}, obtaining
$\displaystyle\frac{\cc_1^h\,T_{n-h}}{(n+1)^h\,h!}[\M]$; this completes the proof of \eqref{n+1 precise}.
In order to prove \eqref{c1 n+1} it is sufficient to consider \eqref{n+1 precise} with $h=n$ (or $h=n-1$).
By taking $h=n-2$, from \eqref{n+1 precise} we have 
\begin{equation}\label{sigma 2}
\cc_1^{n-2}\left(\frac{\cc_1^2+\cc_2}{12}\right)[\M]=N_0 \frac{(n-2)!(n+1)^{n-2}}{n!}\left[ \begin{array}{c} n+1 \\ n-1 \end{array} \right]\,,
\end{equation} 
which, combined with \eqref{c1 n+1} and \eqref{sigma2} proves \eqref{c1c2}.
In order to prove \eqref{gen fct n+1}, observe that, by the above discussion, if $\k0=n+1$ then $\Hil(z)$ is either of degree $n$,
which happens exactly if $N_0\neq 0$, or it is identically zero. In the first case, by Proposition \ref{gen fct hilbert}, $\U(t)$ is of degree
zero and $\U(0)=N_0$, implying \eqref{gen fct n+1}.
\end{proof}

As we will see in the next proposition, the case $\k0=n$ is similar to $\k0=n+1$.
We recall that the Hilbert polynomial of $Q$, the hyperquadric in $\C P^{n+1}$, is given by $\frac{2}{n!}\Big(z+\frac{n}{2}\Big)\prod_{j=1}^{n-1}(z+j)$.
\begin{prop}[$\k0=\mathbf{n}$]\label{cor n}

Let $\ac$ be and $S^1$-space with index $\k0=n$. Let $N_0$ be the number of fixed points with $0$ negative weights.
Then $n\geq 2$ and
\begin{equation}\label{H k0=n}
\Hi(z)= \frac{2\,N_0}{n!}\Big(z+\frac{n}{2}\Big)\prod_{j=1}^{n-1}(z+j)=N_0 \Hil_Q(z)\,, 
\end{equation}
where $\Hil_Q(z)$ is the Hilbert polynomial of $Q$, the hyperquadric in $\C P^{n+1}$. 
Thus for every $h=0,\ldots,n$ we have 
\begin{equation}\label{n precise}
\cc_1^h\, T_{n-h}[\M]=N_0\frac{2\,h!\,n^h}{n!}\Big(\left[ \begin{array}{c} n \\ h \end{array} \right]+\frac{n}{2}\left[ \begin{array}{c} n \\ h+1 \end{array} \right]\Big)\, = N_0 \, \cc_1^h\,T_{n-h}[Q].
\end{equation}
In particular 
\begin{equation}\label{c1 n}
\cc_1^n[\M]=N_0\, 2 n^n
\end{equation}
and
\begin{equation}\label{c1c22}
\cc_1^{n-2}\cc_2[\M]=N_0\,n^{n-2}(n^2-n+2)\,.
\end{equation}
Moreover, the generating function of $\{\Hil(k)\}_{k\in \mathbb{N}}$ is given by
\begin{equation}\label{gen fct n}
\Gen(t)=N_0\frac{1+t}{(1-t)^{n+1}}
\end{equation}
\end{prop}

%%%%%%%%%

\begin{rmk}\label{pos roots n}
From \eqref{gen fct n} and Proposition \ref{gen fct hilbert} we have that in this case $\U(t)=N_0(1+t)$. Thus,
if $N_0\neq 0$, the root of $\U(t)$ is on the unit circle, and 
the zeros of $\Hil(z)$ coincide with the integers greater than $-\k0=-n$ and smaller than $0$, together with $-\frac{n}{2}$,
thus in particular $\Hil(z)$ belongs to $\mathcal{T}_{n}$, and hence its roots are on the canonical strip $\mathcal{S}_{n}$ (see Theorem \ref{RV theorem}).

\end{rmk}

 %%%%%%%%%%%%%%%%%%%%%%%%%%%%%%%%%%%%%%%%%%%%%

\begin{proof}[Proof of Proposition \ref{cor n}]
First of all, for $n=1$ observe that the only compact almost complex manifold supporting a circle action with discrete fixed point set is the
sphere, since such surface must have positive Euler characteristic. In this case $\k0=2$. So we must have $n\geq 2$, and hence $\k0\geq 2$.

The proof of the rest is very similar to that of Proposition \ref{cor n+1}, but we include it here for the sake of completeness.
If $N_0=0$ then all the claims in Proposition \ref{cor n} follow from Corollary \ref{bound on k0} ({\bf ii}). 
Suppose that $N_0\neq 0$. Then by Proposition \ref{properties P} \eqref{1a} we have that $\Hil(z)\not\equiv 0$, and from
Theorem \ref{main theorem} \eqref{H=0 even} and Corollary \ref{extra root -k02} we have that 
$\Hil(z)=\beta(z+\frac{n}{2})\prod_{j=1}^{n-1}(z+j)$. 
In order to determine $\beta$ we can use 
Proposition \ref{properties P} \eqref{1a}, obtaining 
$\beta=\frac{2\,N_0}{n!}$, thus implying \eqref{H k0=n}. The equations in \eqref{n precise} follow easily from observing that
$$
\sigma_{n-h}\Big(1,2,\ldots,n-1,\frac{n}{2}\Big)=\sigma_{n-h}\big(1,2,\ldots,n-1\big)+\frac{n}{2}\sigma_{n-h-1}(1,2,\ldots,n-1)=\left[ \begin{array}{c} n \\ h \end{array} \right]+\frac{n}{2}\left[ \begin{array}{c} n \\ h+1 \end{array} \right]\,.
$$

In order to prove \eqref{c1 n} it is sufficient to consider \eqref{n precise} with $h=n$ (or $h=n-1$).
To prove \eqref{c1c22}, first of all observe that 
$$
\sigma_2(1,2,\ldots,n-1,\frac{n}{2})=\sigma_2(1,2,\ldots,n-1)+\frac{n}{2}\sigma_1(1,2,\ldots,n-1)=\frac{1}{24}n(n-1)(3n^2-n+2)\,,
$$
where the last equality follows from \eqref{sigma1} and \eqref{sigma2}.
Thus if we take $h=n-2$ in \eqref{n precise} we obtain
\begin{align*}
\cc_1^{n-2}\left(\frac{\cc_1^2+\cc_2}{12}\right)[\M]& =N_0\frac{2(n-2)!n^{n-2}}{n!}\sigma_2(1,2,\ldots,n-1,\frac{n}{2})\\
 & = \frac{N_0}{12}n^{n-2}(3n^2-n+2)\,,\\
\end{align*}
and the conclusion follows from \eqref{c1 n}.

In order to prove \eqref{gen fct n}, observe that, by the above discussion, if $\k0=n$ then $\Hil(z)$ is either of degree $n$,
which happens exactly if $N_0\neq 0$, or it is identically zero. In the first case, by Proposition \ref{gen fct hilbert} and Corollary \ref{U palindrom}, $\U(t)$ is a self-reciprocal polynomial of degree
one and $\U(0)=N_0$, thus implying \eqref{gen fct n}.

\end{proof}

From Propositions \ref{cor n+1} and \ref{cor n} we can see that the cases $\k0=n+1$ and $\k0=n$ are very similar, in the sense that
the Hilbert polynomial $\Hil(z)$, as well as the combinations of Chern numbers $\cc_1^h\,T_{n-h}[\M]$, for $h=0,\ldots,n$, and the generating
function $\Gen(t)$ of $\{\Hil(k)\}_{k\in \mathbb{N}}$, are completely determined (see Remark \ref{num of cds}).
\begin{rmk}\label{liham}
In recent work Li \cite{Li} proves that if the $2n$-dimensional manifold $\M$ is symplectic, the $S^1$ action Hamiltonian and $\chi(\M)=n+1$, then having $\k0=n+1$ (resp.\ $\k0=n$)
is equivalent to having the same total Chern class of $\C P^n$ (resp.\ of the Grassmannian of oriented planes in $\R^{n+2}$ with $n$ odd) which, in turns, is equivalent to having the same integral  
cohomology ring of $\C P^n$ (resp.\ the Grassmannian). Thus in particular, under the above hypotheses, all the Chern numbers are `standard', i.e.\ they agree with those of $\C P^n$ (resp.\ of the hyperquadric).
The assumption $\chi(\M)=n+1$ is essential, since it implies the existence of a quasi-ample line bundle (in the sense specified in Remark \ref{hattori rmk}) which in this case is given by the pre-quantization line bundle (see also \cite[Proposition 7.5 (i)]{GoSa}). 
\end{rmk}

In the following we analyse in details the cases $\k0=n-1$ and $\k0=n-2$.
Observe that if $n=1$ the index $\k0$ cannot be zero, since the only compact almost complex surface
that can be endowed with a compatible $S^1$-action with isolated fixed points is the sphere, for which $\k0= 2$.
So in the next proposition it is not restrictive to assume $n\geq 2$ for $\k0=n-1$.

\begin{prop}[$\k0=\mathbf{n-1}$]\label{k0=n-1}
Let $\ac$ be an $S^1$-space of dimension $2n\geq 4$ with index $\k0=n-1$.
\begin{itemize}
\item[(a)]\label{n-1 a} If $N_0\neq 0$ and $\cc_1^n[\M]\neq 0$ then 
\begin{equation}\label{H n-1 1}
\Hil(z)=\frac{4\,N_0}{(n-2)!\big[(n-1)^2-4a\big]}\Big(z^2+(n-1)z+\frac{(n-1)^2}{4}-a\Big)\prod_{j=1}^{n-2}(z+j)\,,
\end{equation}
where $a\in \R$ is not equal to $\frac{(n-1)^2}{4}$. Moreover 
\begin{equation}\label{c1n n-1}
\cc_1^n[\M]=\frac{4\,N_0\,n(n-1)^{n+1}}{(n-1)^2-4a}\,,
\end{equation}
and
\begin{equation}\label{c1c222}
\cc_1^{n-2}\cc_2[\M]=\frac{4N_0(n-1)^{n-2}}{\big[(n-1)^2-4a\big]}\Big[3-12a-6n+\frac{9}{2}n^2-2n^3+\frac{n^4}{2}\Big]\,.
\end{equation}
\item[(b)]\label{n-1 b} If $N_0\neq 0$ and $\cc_1^n[\M]= 0$ then
\begin{equation}\label{H n-1 2}
\Hil(z)=\frac{N_0}{(n-2)!}\prod_{j=1}^{n-2}(z+j)\,,
\end{equation}
%%%
and 
\begin{equation}\label{c1c2 n-1}
\cc_1^{n-2}\cc_2[\M]=12\, N_0 (n-1)^{n-2}\,.
\end{equation}
\end{itemize}

Moreover, in \emph{(a)} and \emph{(b)}, the generating function of $\{\Hil(k)\}_{k\in \mathbb{N}}$ is given by 
\begin{equation}\label{gen fct n-1 1}
\Gen(t)=N_0\frac{1+b\,t+\,t^2}{(1-t)^{n+1}}  
\end{equation}
where $b\in \mathbb{Q}$ is such that $b\,N_0\in \Z$ and
\begin{equation}\label{c1n n-1 b}
\cc_1^n[\M]=N_0(b+2)(n-1)^n\,,
\end{equation}
\begin{equation}\label{c1c2 n-1 b}
\cc_1^{n-2}\cc_2[\M]=N_0(n-1)^{n-2}\big[12+\frac{(b+2)n(n-3)}{2}\big]\,.
\end{equation}
(Thus case \emph{(b)} corresponds to taking $b=-2$.)

\begin{itemize} 
\item[(c)]\label{n-1 c}
If $N_0=0$ then 
\begin{equation}\label{H n-1 3}
\Hil(z)=\gamma \prod_{j=0}^{n-1}(z+j)\,,
\end{equation}
where $\gamma=\frac{1}{(n-1)^n n!} \cc_1^n[\M]$.
\end{itemize}
Moreover, the generating function of $\{\Hil(k)\}_{k\in \mathbb{N}}$ is given by 
\begin{equation}\label{gen fct n-1 3}
\Gen(t)=\gamma\,n!\frac{t}{(1-t)^{n+1}}\,. 
\end{equation}
\end{prop}
\begin{rmk}\label{integrality a}
Observe that the value of $a$ in \eqref{c1n n-1} cannot be arbitrary, since the following fraction
$$
\frac{4N_0\,n(n-1)}{(n-1)^2-4a}
$$
must be an integer. This follows from the fact that, modulo torsion, $\cc_1=(n-1)\eta_0$ for some $\eta_0\in H^2(\M;\Z)$,
and hence $\frac{\cc_1^n[\M]}{(n-1)^n}$ must be an integer.
\end{rmk}

The following corollary is a straightforward consequence of Proposition \ref{k0=n-1}
 \begin{corollary}\label{pos roots n-1}
 Under the same hypotheses of Proposition \ref{k0=n-1}, we have that:\\
 - If $N_0\neq 0$ then  
 \begin{itemize}
 \item[(1)] The roots of $\Hil(z)$ belong to the canonical strip $\mathcal{S}_{n-1}$ if and only if $\cc_1^n[\M]\geq 0$, or equivalently if and only if $b\geq -2$.
 \item[(2)]  $\Hil(z)$ belongs to $\mathcal{T}_{n-1}$ if and only if $\;\;\;0\leq \cc_1^n[\M]\leq 4N_0 n(n-1)^{n-1}$, or equivalently if and only if $\;\;\;-2\leq b \leq 2\displaystyle\frac{n+1}{n-1}$.
 \end{itemize}
 - If $N_0=0$ then the roots of $\Hil(z)$ do not belong to $\mathcal{S}_{n-1}$.
 \end{corollary}
 As a result of the analysis carried out when $\k0=n-1$, we can strengthen Theorem \ref{RV theorem}.
 \begin{corollary}\label{RV1}
 Under the same hypotheses of Proposition \ref{k0=n-1}, assume that $N_0=1$ and $n>5$. Then $\Hil(z)$ belongs to $\mathcal{T}_{n-1}$ if and only if $\U(t)$ has its roots on the unit circle. 
 \end{corollary}
 \begin{proof}
 If $N_0=1$ then by Proposition \ref{k0=n-1} we know that $b$ is an integer.  
If $n>5$, from Corollary \ref{pos roots n-1} we can see that $\Hil(z)$ belongs to
$\mathcal{T}_{n-1}$ if and only if $-2\leq b\leq 2$. Since $b$ is an integer, for all such values of $b$ the polynomial $\U(t)=1+bt+t^2$ has its roots on the unit circle.  
 \end{proof}
\begin{rmk}\label{RV n-1}
For $2\leq n\leq 5$, we have that $2\displaystyle\frac{n+1}{n-1}\geq 3$; however
for $b\geq 3$, the roots of $\U(t)$ are not on the unit circle. So for $2\leq n\leq 5$,
there may exist manifolds whose associated Hilbert polynomial belongs to $\mathcal{T}_{n-1}$, but the corresponding $\U(t)=1+bt+t^2$ does not have its roots on the unit circle: consider for example the Fano threefold $V_5$
in Example \ref{examples 6} (3), for which $b=3$ and the corresponding Hilbert polynomial is given by $\Hil_{V_5}(z)=\frac{1}{6}\big[5z^2+10z+6\big](z+1)$. 
\end{rmk}

\begin{proof}[Proof of Proposition \ref{k0=n-1}]
(a) If $N_0\neq 0$ then, by Proposition \ref{properties P} \eqref{1a} we have that $\Hil(z)\not\equiv 0$. Moreover by \eqref{Hilbert pol}, if $\cc_1^n[\M]\neq 0$ then $\deg(\Hil)=n$.
By Theorem \ref{main theorem} \eqref{H=0 even}, if $n\geq 3$ $\Hil(z)$ has roots $-1,-2,\ldots, -n+2$. By Corollary 
\ref{property roots}, the remaining two roots belong to $\mathcal{C}_{n-1}$ and, by Proposition \ref{properties P} \eqref{3a},
they are of the form $-\frac{n-1}{2}-x$, $-\frac{n-1}{2}+x$. Moreover $a:=x^2\neq \frac{(n-1)^2}{4}$
since by Proposition \ref{properties P} \eqref{1a} and \eqref{3a}, $\Hil(0)=N_0$, $\Hil(-n+1)=(-1)^nN_0$ and by assumption $N_0\neq 0$.  Thus $\Hil(z)=\alpha \Big(z^2+(n-1)z+\frac{(n-1)^2}{4}-a\Big)\prod_{j=1}^{n-2}(z+j)$,
where $\alpha\in \R$ can be found by imposing $\Hil(0)=N_0$, obtaining \eqref{H n-1 1}.
Equations \eqref{c1n n-1} and \eqref{c1c222} come from combining \eqref{Hilbert pol} with \eqref{H n-1 1}.

(b) If $N_0\neq 0$ and $\cc_1^n[\M]=0$ then, by Proposition \ref{properties P} \eqref{1a} we have that $\Hil(z)\not\equiv 0$
and, by \eqref{Hilbert pol}, $\deg(\Hil)\leq n-2$.
By Theorem \ref{main theorem}, if $n\geq 3$ $\Hil(z)$ has $n-2$ roots given by $-1,-2,\ldots,-n+2$;
moreover if $n=2$ it must be a non-zero constant polynomial. Thus $\Hil(z)$ has degree $n-2$ and it is of the form
$\Hil(z)=\beta \prod_{j=1}^{n-2}(z+j)=\beta\sum_{h=0}^{n-2}z^h\sigma_{n-h-2}(1,2,\ldots,n-2)$. By Proposition \ref{properties P} \eqref{1a} we have $\beta=\frac{N_0}{(n-2)!}$,
and \eqref{H n-1 2} follows.
Equation \eqref{c1c2 n-1} can be obtained from \eqref{H n-1 2} and \eqref{ah}
by taking $h=n-2$.

In order to prove \eqref{gen fct n-1 1} for $\cc_1^n[\M]\neq 0$, observe that since $\deg(\Hil)=n$, $N_0\neq 0$ and $\k0=n-1$, from Proposition \ref{gen fct hilbert} and Corollary \ref{U palindrom} it follows
that $\U(t)=N_0(1+b\,t+t^2)$ for some $b\in \R$. Thus we have that 
$$
\Gen(t)=N_0 \frac{1+b\,t+t^2}{(1-t)^{n+1}} = N_0\sum_{k\geq 0}\left[ \binom{n+k-2}{n}+b\binom{n+k-1}{n}+\binom{n+k}{n}\right]t^k\,, 
$$
and by definition of $\Gen(t)$ we have that $N_0(b+n+1)=\Hil(1)$.  Since $\Hil(1)$ is an integer, it follows that $b\,N_0$ must be an integer.
Moreover, by \eqref{H n-1 1} we have that $\displaystyle\frac{\Hil(1)}{N_0}=\frac{4(n-1)\big[n+\frac{(n-1)^2}{4}-a\big]}{\big[(n-1)^2-4a\big]}=b+n+1$, thus obtaining $b$ in terms of $a$, and the expressions of $\cc_1^n[\M]$ and $\cc_1^{n-2}\cc_2[\M]$ in terms of $b$ follow from \eqref{c1n n-1} and \eqref{c1c222}.

The proof of \eqref{gen fct n-1 1} when $\cc_1^n[\M]=0$ also follows from Proposition \ref{gen fct hilbert}, and the details are left to the reader.

(c) If $N_0=0$ then, by Proposition \ref{properties P} \eqref{1a} and \eqref{3a}, and Theorem \ref{main theorem} \eqref{H=0 even}, $\Hil(z)$ has $n$ roots given by $0,-1,-2,\ldots,-n+1$. If $\cc_1^n[\M]=0$ then
by \eqref{Hilbert pol} and \eqref{ah} we have that $\deg(\Hil)\leq n-2$, hence $\Hil(z)\equiv 0$ and \eqref{H n-1 3} follows.
Otherwise $\Hil(z)=\gamma \prod_{j=0}^{n-1}(z+j)$ where the expression for $\gamma$ can be obtained by using
\eqref{Hilbert pol}, imposing that $a_n=\gamma$.

The proof of \eqref{gen fct n-1 3} follows easily from Proposition \ref{gen fct hilbert}, and the details are left to the reader.

\end{proof}

Proposition \ref{k0=n-1} implies that the Chern numbers $\cc_1^n[\M]$ and $\cc_1^{n-2}\cc_2[\M]$ are related
by the following formula.
\begin{corollary}\label{relation c122}
Under the same hypotheses of Proposition \ref{k0=n-1} we have that 
$$
\cc_1^{n-2}\cc_2[\M]-\frac{n(n-3)}{2(n-1)^2}\cc_1^n[\M] = 12 N_0(n-1)^{n-2}
$$
\end{corollary}
\begin{proof}
When $N_0\neq 0$ the claim follows from \eqref{c1n n-1 b} and \eqref{c1c2 n-1 b}.

If $N_0=0$ and $\cc_1^n[\M]=0$ then from \eqref{H n-1 3} we have $\Hil(z)\equiv 0$, which, by \eqref{ah} implies that 
$$a_{n-2}=\frac{1}{12(n-1)^{n-2}(n-2)!}\big(\cc_1^n + \cc_1^{n-2}\cc_2\big)[\M]=0\,,$$
thus implying $\cc_1^{n-2}\cc_2[\M]=0$, and the claim follows. 

Otherwise, if $N_0\neq 0$ and $\cc_1^n[\M]\neq 0$, from \eqref{H n-1 3} and \eqref{sigma 2} we have that 
$a_{n-2}$ is 
\begin{equation}\label{an-2}
a_{n-2}=\gamma \left[ \begin{array}{c} n \\ n-2 \end{array} \right]= \gamma \frac{(3n-1)n(n-1)(n-2)}{24}\,,
\end{equation}
where $\gamma=\frac{1}{(n-1)^n n!}\cc_1^n[\M]$,
and the claim follows from comparing the general expression of $a_{n-2}$ with \eqref{an-2}.
\end{proof}

As it will be proved in Prop.\ \ref{dim 4}, if $\ac$ is an $S^1$-space of dimension $4$, the index $\k0$ cannot be zero.
Hence it is not restrictive to assume $n\geq 3$ for $\k0=n-2$.
\begin{prop}[$\k0=\mathbf{n-2}$]\label{k0=n-2}
Let $\ac$ be an $S^1$-space of dimension $2n\geq 6$ with index $\k0=n-2$.
\begin{itemize}
\item[(a)]\label{n-2 a} If $N_0\neq 0$ and $\cc_1^n[\M]\neq 0$ then 
\begin{equation}\label{H n-2 1}
\Hil(z)=\frac{4\,N_0}{(n-2)!\big[(n-2)^2-4a\big]}\Big(2z+n-2\Big)\Big(z^2+(n-2)z+\frac{(n-2)^2}{4}-a\Big)\prod_{j=1}^{n-3}(z+j)\,,
\end{equation}
where $a\in \R$ is not equal to $\frac{(n-2)^2}{4}$. Moreover
\begin{equation}\label{c1n n-2}
\cc_1^n[\M]=\frac{8\,N_0\,n(n-1)(n-2)^n}{(n-2)^2-4a}\,,
\end{equation}
and
\begin{equation}\label{c1c2 n-2 2 rmk}
\cc_1^{n-2}\cc_2[\M]=\frac{4N_0(n-2)^{n-2}(24-24a-30n+17n^2-6n^3+n^4)}{(n-2)^2-4a}\,. 
\end{equation}
\\

\item[(b)]\label{n-2 b} If $N_0\neq 0$ and $ \cc_1^n[\M]= 0$ then
\begin{equation}\label{H n-2 2}
\Hil(z)=\frac{N_0}{(n-2)!}\Big(2z+n-2\Big)\prod_{j=1}^{n-3}(z+j)\,,
\end{equation}
and
%%%%
\begin{equation}\label{c1c2 n-2}
\cc_1^{n-2}\cc_2[\M]=24\, N_0 (n-2)^{n-2}\,.
\end{equation}
\end{itemize}
Moreover, in \emph{(a)} and \emph{(b)}, the generating function of $\{\Hil(k)\}_{k\in \mathbb{N}}$ is given by 
\begin{equation}\label{gen fct n-2}
\Gen(t)=N_0\frac{1+b\,t+b\,t^2+t^3}{(1-t)^{n+1}}  
\end{equation}
where $b$ is such that $b\,N_0$ is an integer and 
\begin{equation}\label{c1n b}
\cc_1^n[\M]=2N_0(b+1)(n-2)^n\,,
\end{equation}
\begin{equation}\label{c1c2 b}
\cc_1^{n-2}\cc_2[\M]=N_0(n-2)^{n-2}\big[24+(b+1)(n-2)(n-3)\big]\,,
\end{equation}
and case \emph{(b)} corresponds to taking $b=-1$.
\begin{itemize}
\item[(c)]\label{n-2 c}
If $N_0=0$ then 
\begin{equation}\label{H n-2 3}
\Hil(z)=\gamma \Big(z+\frac{n-2}{2}\Big)\prod_{j=0}^{n-2}(z+j)\,,
\end{equation}
where $\gamma=\frac{1}{(n-2)^n n!}\cc_1^n[\M]$.
\end{itemize}
Moreover, the generating function of $\{\Hil(k)\}_{k\in \mathbb{N}}$ is given by 
\begin{equation}\label{gen fct n-2 3}
\Gen(t)=\frac{\gamma}{2}n!\frac{t+t^2}{(1-t)^{n+1}}\,.  
\end{equation}

\end{prop}

\begin{rmk}\label{integrality a2}
The same comment in Remark \ref{integrality a} applies here: the value of $a$ cannot be arbitrary, since the following fraction
$$
\frac{8N_0\,n(n-1)}{(n-2)^2-4a}
$$
must be an integer.
\end{rmk}

The following corollary is very similar to Corollary \ref{pos roots n-1}, and is a straightforward consequence of Proposition \ref{k0=n-2}
 \begin{corollary}\label{pos roots n-2}
 Under the same hypotheses of Proposition \ref{k0=n-2}, we have that:\\
 - If $N_0\neq 0$ then  
 \begin{itemize}
 \item[(1)] The roots of $\Hil(z)$ belong to the canonical strip $\mathcal{S}_{n-2}$ if and only if $\cc_1^n[\M]\geq 0$, or equivalently if and only if $b\geq -1$.
 \item[(2)] $\Hil(z)$ belongs to $\mathcal{T}_{n-2}$ if and only if $\;\;\;0\leq \cc_1^n[\M]\leq 8N_0 n(n-1)(n-2)^{n-2}$, or equivalently if and only if $\;\;\;-1\leq b \leq \displaystyle\frac{3n^2-4}{(n-2)^2}$.
 \end{itemize}
 - If $N_0=0$ then the roots of $\Hil(z)$ do not belong to $\mathcal{S}_{n-2}$.
 \end{corollary}
In analogy with Corollary \ref{RV1}, we have the following:
\begin{corollary}\label{RV2}
 Under the same hypotheses of Proposition \ref{k0=n-2}, assume that $N_0=1$ and $n>14$. Then $\Hil(z)$ belongs to $\mathcal{T}_{n-2}$ if and only if $\U(t)$ has its roots on the unit circle. 
 \end{corollary}
\begin{proof}
If $N_0=1$ then by Proposition \ref{k0=n-2}, we know that $b$ is an integer.  If $n>14$, from Corollary \ref{pos roots n-2} we can see that $\Hil(z)$ belong to
$\mathcal{T}_{n-2}$ if and only if $-1\leq b\leq 3$. Since $b$ is an integer, for all such values of $b$ the polynomial $\U(t)=1+bt+bt^2+t^3$ has its roots on the unit circle.
\end{proof}
\begin{rmk}\label{RV n-2}
For $3\leq n\leq 14$, we have that $\displaystyle\frac{3n^2-4}{(n-2)^2}\geq 4$; however
for $b\geq 4$, the roots of $\U(t)=1+bt+bt^2+t^3$ are not on the unit circle. In conclusion, we can say that for $3\leq n\leq 14$,
there may exist manifolds whose associated Hilbert polynomial belongs to $\mathcal{T}_{n-2}$, but the corresponding $\U(t)$ does not have its roots on the unit circle: consider for example the Fano threefold $V_{22}$
in Example \ref{examples 6-1} (2), for which $b=10$ and the corresponding Hilbert polynomial is given by $\Hil_{V_{22}}(z)=\frac{1}{6}\big[11z^2+11z+6\big](2z+1)$.
\end{rmk}

\begin{proof}[Proof of Proposition \ref{k0=n-2}]

The proof of this Proposition is very similar to that of Proposition \ref{k0=n-1}, and here we only sketch the first part.
(a)  If $N_0\neq 0$ then, by Proposition \ref{properties P} \eqref{1a} we have that $\Hil(z)\not\equiv 0$. Moreover by \eqref{Hilbert pol}, if $ \cc_1^n[\M]\neq 0$ then $\deg(\Hil)=n$.
By Theorem \ref{main theorem} \eqref{H=0 even}, for $n\geq 4$ $\Hil(z)$ has roots $-1,-2,\ldots, -n+3$. By Corollary \ref{extra root -k02} one of the remaining three roots is $-\frac{n-2}{2}$. By Corollary \ref{property roots} the remaining two roots are on $\mathcal{C}_{n-2}$, and
by Proposition \ref{properties P} \eqref{3a} they are of the form $-\frac{n-2}{2}-x$, $-\frac{n-2}{2}+x$, for some $x\in \R$.
Moreover $a:=x^2\neq \frac{(n-2)^2}{4}$
since by Proposition \ref{properties P} \eqref{1a} and \eqref{3a}, $\Hil(0)=N_0$, $\Hil(-n+2)=(-1)^nN_0$ and by assumption $N_0\neq 0$. 
It follows that the Hilbert polynomial is of the form
$$
\Hil(z)=\alpha \Big(2z+n-2\Big)\Big(z^2+(n-2)z+\frac{(n-2)^2}{4}-a\Big)\prod_{j=1}^{n-3}(z+j)\,,
$$
where $\alpha$ can be found by imposing $\Hil(0)=N_0$, thus obtaining \eqref{H n-2 1}. 
The rest of the proof is left to the reader. 
\end{proof}

Similarly to the case $\k0=n-1$, Proposition \ref{k0=n-2} implies that the Chern numbers $\cc_1^n[\M]$ and
$\cc_1^{n-2}\cc_2[\M]$ are related by the following formula.
\begin{corollary}\label{relation c122 2}
Under the same hypotheses of Proposition \ref{k0=n-2} we have that 
$$
\cc_1^{n-2}\cc_2[\M]-\frac{n-3}{2(n-2)}\cc_1^n[\M]= 24 N_0 (n-2)^{n-2}
$$
\end{corollary}
\begin{proof}
The proof of this Corollary is very similar to that of Corollary \ref{relation c122}, and the details are left to the reader.
\end{proof}

As a consequence of the analysis of $\Hil(z)$ when the index $\k0$ is $n-2$ or $n$ we have the following
\begin{corollary}\label{cor even chern classes}
Let $\ac$ be an $S^1$-space with $N_0\neq 0$. Assume the index satisfies either $\k0=n$, or $\k0=n-2$ and $n\geq 3$.
Then the Chern numbers $\cc_1^n[\M]$ and $\cc_1^{n-2}\cc_2[\M]$ are always \emph{even}.
\end{corollary}
\begin{proof}
When $\k0=n$ the claim follows from Proposition \ref{cor n} \eqref{c1 n} and \eqref{c1c22},
and when $\k0=n-2$ it follows from Proposition \ref{k0=n-2} \eqref{c1n b} and \eqref{c1c2 b}.
\end{proof}

The case in which $\k0={\bf n-3}$, where $n\geq 4$, is not analysed in details here. However we would like to make some remarks about it when $N_0\neq 0$
and $\deg(\Hil)=n$, i.e.\;$ \cc_1^n[\M]\neq 0$.
First of all, observe that this is the first case in which the roots of $\Hil(z)$ may not belong to $\mathcal{C}_{\k0}$ (see Corollary \ref{position roots}). From Theorem \ref{main theorem} \eqref{H=0 even}, the roots of $\Hil(z)$ are $-1,-2,\ldots,-n+4$ (if $n>4$), plus four additional roots $z_1,z_2,z_3,z_4$. If the remaining four roots don't belong to $\mathcal{C}_{\k0}$, 
from the properties of $\Hil(z)$ they must be of the form $-\frac{n-3}{2}\pm a\pm {\bf i}\, b$, for some $a,b\in \R\setminus \{0\}$, thus obtaining that
\begin{equation}\label{k0=n-3}
\Hil(z)=\alpha \prod \Big(z+\frac{n-3}{2}\pm a\pm {\bf i}\,b\Big)\prod_{j=1}^{n-4}(z+j)\,.
\end{equation}
From the expression of $a_n$ in \eqref{ah} and Proposition \ref{properties P} \eqref{1a} it follows that 
\begin{equation}\label{cassini}
\Big[\Big(\frac{n-3}{2}-a\Big)^2+b^2\Big]\Big[\Big(\frac{n-3}{2}+a\Big)^2+b^2\Big]=\frac{N_0\,n!\,(n-3)^n}{\,(n-4)!\,\cc_1^n[\M]}\,,
\end{equation}
which implies that $\cc_1^n[\M]>0$. Moreover, for a fixed value of $\cc_1^n[\M]$, the four roots $z_1,\ldots,z_4$ 
belong to the \emph{Cassini oval}\footnote{We recall that a \emph{Cassini oval} is a quartic plane curve given by the locus of points in $\R^2\simeq \C$ satisfying the equation $$\mathrm{d}(p,q_1)\,\mathrm{d}(p,q_2)=d^2\,,$$ where $d\neq 0$. The points $q_1$ and $q_2$ are called the \emph{foci} of the Cassini oval.} of equation
\begin{equation}\label{cassini eq}
\mathrm{d}(p,0)\,\mathrm{d}(p,-n+3)=\sqrt{\frac{N_0\,n!\,(n-3)^n}{\,(n-4)!\,\cc_1^n[\M]}}
\end{equation}
where $\mathrm{d}(p,q)$ denotes the Euclidean distance from $p$ to $q$, with $p,q\in\R^2\simeq \C$, and the foci of this oval are the points $0$ and $-n+3$ (see Figure \ref{Fig:cassini}).

%%%%%%%%%%%%%%%%%%%%%%

\begin{figure}[h!]
\begin{center}
\epsfxsize=\textwidth
\leavevmode
\includegraphics[width=3.5in]{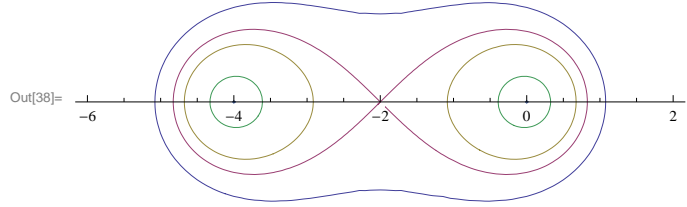}
\end{center}
\caption{Examples of \emph{Cassini ovals} of equation $\mathrm{d}(p,q_1)\,\mathrm{d}(p,q_2)=d^2$ with foci $q_1=(0,0)$ and $q_2=(-4,0)$ for different values of $d$.
The curve passing through the origin is called the \emph{lemniscate of Bernoulli}, and is obtained for $d=4$.}
\label{Fig:cassini}
\end{figure}

%%%%%%%%%%%%%%%%%%%%%

\medskip
\subsection{Conclusions on Hamiltonian and non-Hamiltonian actions.}
As an application of the results obtained before, we conclude the section with the proof of Theorem \ref{nHam-char}. Observe that for $n=1$ and $n=2$ there do not exist symplectic non-Hamiltonian circle actions with nonempty discrete fixed point sets: for $n=1$ the only compact surface admitting such a symplectic circle action is a sphere, hence the action is Hamiltonian; for $n=2$ the assertion was proved by McDuff in \cite[Proposition 2]{MD1}.
%%%%
\begin{proof}[Proof of Theorem \ref{nHam-char}]
We recall that in the symplectic case $N_0$ can be either $0$ or $1$, and it is $0$ exactly if the action is non-Hamiltonian (see
Lemma \ref{N0 1}). 
Then the claims in (I) follow from Corollary \ref{bound on k0 s} ({\bf i'}), those
in (II) and (III) from Propositions \ref{cor n+1} and \ref{cor n}, and those in (IV) and (V) from Corollaries \ref{relation c122} and \ref{relation c122 2}.
\end{proof}
\begin{rmk}\label{other comb}
Observe that 
by Propositions \ref{cor n+1} and \ref{cor n}, when $\k0=n+1$ or $\k0=n$ the action is Hamiltonian if and only if \emph{all} the combinations of 
Chern numbers $\cc_1^h\,T_{n-h}[\M]$ do not vanish, for $h=0,\ldots,n$.
\end{rmk}
%%%%%%%%%%%%%%%%%%%%%%%%%%%%%%%%%%%%%%%%%%%%%%%%%%%%%%%%%%%%%%%%%%%%%%%%
\section{Examples: low dimensions of $(\M,\J)$}\label{examples}

In this section, we study some consequences of the results previously obtained for
$n\leq 4$.
In particular we prove that when $\k0=n$ or $n+1$ then \emph{all the Chern numbers of $(\M,\J,S^1)$ can be expressed as a linear combination of the $N_j$'s}, where $N_j$ denotes the number of fixed points with exactly $j$ negative weights. In the Hamiltonian category, this amounts to saying that \emph{all the Chern numbers of $(\M,\omega, S^1)$ can be expressed as linear combinations of the Betti numbers of }$\M$ (see \eqref{bi=Ni}).

The most obvious Chern number that can always be written in terms of the $N_j$'s 
is  $\cc_n[\M]$. In fact, by definition of the $N_j$'s and $\cc_n[\M]=|\M^{S^1}|$, we have 
\begin{equation}\label{eq cn}
\cc_n[\M]=\sum_{j=0}^n N_j\,.
\end{equation}
In \cite{GoSa}, Godinho and the author proved that the Chern number $\cc_1\cc_{n-1}[\M]$ can also be expressed in terms of the $N_j$'s. 
We recall its explicit expression in the following
\begin{theorem}[\cite{GoSa} Theorem 1.2]\label{nostro}
 Let $(\M,\J, S^1)$ and $N_j$ be as above. Then 
 \begin{equation}\label{c1cn-1}
 \cc_1\cc_{n-1}[\M]=\sum_{j=0}^n N_j \Big[6j(j-1)+\frac{5n-3n^2}{2}\Big]\,.
 \end{equation}

\end{theorem}

Suppose that $\ac$ is an $S^1$-space
of (real) dimension $2$. As also observed before, since we are requiring isolated fixed points,
such a space must be a $2$-sphere, obtaining 
 $\k0=2$, $\Hil(z)=1+z$ and $\cc_1[S^2]=2$. 

\subsection{$\mathbf{\dim(\M)=4}$} First of all, observe that by \eqref{NiN} and \eqref{eq cn} we have
\begin{equation}\label{c2 dim 4}
 \cc_2[\M]=2N_0+N_1\,.
\end{equation}
Moreover, 
by \eqref{NiN} and Theorem \ref{nostro} \eqref{c1cn-1}, for $n=2$ it follows that
\begin{equation}\label{c1 n=2}
\cc_1^2[\M]=10 N_0 - N_1\,.
\end{equation}
Thus in dimension $4$ all the Chern numbers can be expressed as a linear combination of the $N_j$'s (independently on $\k0$).
\begin{rmk}\label{pos 1}
Observe that the necessary condition $\cc_1^2+\cc_2[\M]\equiv 0 \mod{12}$, which must hold for
any compact almost complex manifold, for $S^1$-spaces becomes 
$\cc_1^2+\cc_2[\M]=12 N_0$ (it is equivalent to saying that the Todd genus is $N_0$). Hence, 
for $\ac$, the combination of Chern numbers $\cc_1^2+\cc_2[\M]$ must be a \emph{non-negative} multiple of $12$.
\end{rmk}
The following corollary is an easy consequence of the results obtained before, applied to the symplectic category:
\begin{corollary}\label{geo s}
Let $(\M,\omega)$ be a compact, connected symplectic manifold of dimension $4$ that can be endowed with a symplectic circle action with isolated fixed points.
Then 
\begin{equation}\label{c1 n2 h}
(\cc_1^2[\M], \cc_2[\M])=(10-b_2(\M), 2+ b_2(\M))\,.
\end{equation}
Moreover, any pair of integers $(p,q)$ satisfying $p+q=12$ and $p\leq 9$ can be realized as the pair of Chern numbers $(\cc_1^2[\M], \cc_2[\M])$
of a compact, connected symplectic manifold $\M$ of dimension $4$ supporting a symplectic circle action with isolated fixed points.
\end{corollary}
\begin{proof}
Given $(\M,\omega,S^1)$ of dimension $4$, a theorem of McDuff \cite{MD1} implies that the action is Hamiltonian, and as a consequence of \eqref{c1 n=2}, Lemma \ref{N0 1}
and \eqref{bi=Ni} we obtain \eqref{c1 n2 h}.
The second assertion follows from observing that $b_2(\M)$ is positive, it is at least one (consider $(\C P^2,\omega_{F},S^1)$, see Example \ref{ch}), and can be arbitrarily large: to obtain
$(\M,\omega,S^1)$ with $b_2(\M)=k$, it is sufficient to perform $(k-1)$ times an $S^1$-equivariant blow-up on $(\C P^2,\omega_{F},S^1)$.
\end{proof}
\begin{exm}\label{ch}[{\bf The complex projective space and Hirzebruch surfaces}]
Consider $\C P^2$ endowed with a multiple of the Fubini-Study form $\omega_F$, and `standard' $S^1$-action, namely $S^1$ is a circle subgroup of a 2-dimensional torus $\T^2$ acting on $\C P^2$
in a toric way. Thus the $S^1$-action is given by
$\alpha \cdot [z_0:z_1:z_2]=[z_0:\alpha^l z_1:\alpha^{l+m}z_2]$
for every  $\alpha\in S^1$ (where $l$ and $m$ are non-zero, coprime integers) it has three fixed points and is Hamiltonian. Note that the minimal Chern number of $\C P^2$  is $3$.
We denote this $S^1$-space by $(\C P^2,\lambda \,\omega_F,S^1)_{l,m}$, where $\lambda \in \R_{>0}$.

For every $k\in \Z$, let $\mathcal{H}_k$ be the Hirzebruch surface: $\{([z_0:z_1:z_2],[w_1:w_2])\in \C P^2 \times \C P^1\mid z_1\,w_2^k=z_2\,w_1^k\}$, endowed
with symplectic form $\widetilde{\omega}$ induced by multiples of the Fubini-Study forms on $\C P^2$ and $\C P^1$.  We can give each $\mathcal{H}_k$ an $S^1$-action, defined by:
 $\alpha\cdot ([z_0:z_1:z_2],[w_1:w_2])=([\alpha^l z_0:z_1:\alpha^{k\,m}z_2)],[w_1,\alpha^m w_2])$, where $l$ and $m$ are non-zero, coprime integers. This action has $4$ fixed points and is Hamiltonian.
 We denote these $S^1$-spaces by $(\mathcal{H}_k,\widetilde{\omega},S^1)_{l,m}$.
 Note that the minimal Chern number of $\mathcal{H}_k$ is $1$ if $k$ is odd and $2$ if $k$ is even, and $\mathcal{H}_k$ is respectively called an \emph{odd} or \emph{even} Hirzebruch surface.
 
\end{exm}
\begin{rmk}\label{minimal spaces}
The examples above are exactly the \emph{minimal spaces} obtained in the classification of $(\M,\omega,S^1)$ of dimension $4$ (if the fixed point set is not discrete, there is an additional class of minimal spaces
given by $\C P^1$-bundles over Riemann surfaces of genus $g\geq 1$), see \cite{AH,Au} and \cite{K}. More precisely, in \cite{K} Karshon proves that every $(\M,\omega,S^1)$ is equivariantly symplectomorphic to a symplectic
$S^1$-space obtained from $(\C P^2,\lambda \,\omega_F,S^1)_{l,m}$ or
 $(\mathcal{H}_k,\widetilde{\omega},S^1)_{l,m}$ (for suitable $\lambda, l,m,k$ as above) by a sequence of $S^1$-equivariant blow-ups at fixed points.\footnote{Note that the blow-up of $(\C P^2,\lambda \,\omega_F,S^1)_{l,m}$
 at one fixed point is an odd Hirzebruch surface.} 
 \end{rmk}
\begin{rmk}\label{ci 2}
Observe that for every $S^1$-space $(\M,\omega,S^1)$ of dimension $4$ the following inequality holds:
\begin{equation}\label{ci 12}
\cc_1^2[\M]\leq 3\cc_2[\M]\,.
\end{equation}
Indeed, Corollary \ref{geo s} implies that \eqref{ci 12} is equivalent to $b_2(\M)\geq 1$.
Note that \eqref{ci 12} was conjectured by Van de Ven \cite{V}
and proved by Miyaoka \cite{Mi}
for (complex) surfaces of general type. 

The following question is then natural:
\begin{question}\label{inac?}
Let $\ac$ be an $S^1$-space. Does inequality \eqref{ci 12} hold? 
\end{question}
\end{rmk}
By \eqref{c2 dim 4} and \eqref{c1 n=2}, proving inequality \eqref{ci 12} for $\ac$ is equivalent to proving that for every such space $N_0\leq N_1$. 
The next proposition 
implies that the answer to question \ref{inac?} is `yes' for all $4$-dimensional $S^1$-spaces whose index is not one.

\begin{prop}\label{dim 4}
Let $(\M,\J,S^1)$ be an $S^1$-space of dimension $4$, and let
$\k0$, $\Hil(z)$ and the $N_j$'s be defined as before.
Then $N_0, N_1$ and $N_2$ are all non-zero, the first Chern class $\cc_1$ is not a torsion element in $H^2(\M;\Z)$, and $\k0\in \{1,2,3\}$. Moreover 
\begin{itemize}
\item[(a)]If $\k0=3$ then
\begin{equation}\label{dim4 1}
N_0=N_1=N_2,\;\;\;\quad \cc_1^2[\M]=9N_0\;\;\;\quad \mbox{and}\;\;\; \quad \Hil(z)=\frac{N_0}{2}(z+1)(z+2).
\end{equation}
\item[(b)] If $\k0=2$ then
\begin{equation}\label{dim4 2}
2N_0=N_1=2N_2,\;\;\;\quad \cc_1^2[\M]=8N_0\;\;\;\quad \mbox{and}\;\;\; \quad \Hil(z)=N_0(z+1)^2. 
\end{equation}
\end{itemize}
$\;$\\
Given $(\M,\omega,S^1)$ of dimension $4$ we have that
\begin{itemize}
 \item[(a')] $\k0=3$ if and only if there exists $\lambda>0$ and coprime integers $l,m$ such that $(\M,\omega,S^1)$ is equivariantly symplectomorphic to $(\C P^2,\lambda\,\omega_F,S^1)_{l,m}$.
\item[(b')] $\k0=2$ if and only if there exists coprime integers $l,m$, an even $k\in \Z$ and a symplectic form $\widetilde{\omega}$ on $\mathcal{H}_k$ such that $(\M,\omega,S^1)$ is equivariantly symplectomorphic to $(\mathcal{H}_k,\widetilde{\omega},S^1)_{l,m}$.
\end{itemize}
\end{prop}

\begin{proof}
Let $p$ be a fixed point, and $e^{S^1}(p)\in H_{S^1}^2(\{p\};\Z)=\Z[x]$ the equivariant Euler class of the normal bundle at $p$, which is simply given by $w_{1p}w_{2p}x$, where $w_{1p}$ and $w_{2p}$ are the weights of the isotropy $S^1$-action at $p$. By the ABBV formula (Thm.\ \ref{abbv formula}) we must have
\begin{equation}\label{ABBV}
\sum_{p\in \M^{S^1}}\frac{1}{e^{S^1}(p)}=1[\M]=0\,.
\end{equation} 
So it follows that $\M^{S^1}$ must contain points whose product of the corresponding weights is positive, as well as those for which it is negative. Thus $N_0+N_2\neq 0$ which, together with \eqref{NiN}, implies that $N_0$ and $N_2$ are non-zero,
 and $N_1\neq 0$ . From Lemma \ref{c1 N0} (a2) it follows that $\cc_1$ is not a torsion element in $H^2(\M;\Z)$, and by Corollary \ref{bound on k0} ({\bf i}) that $\k0\in \{1,2,3\}$.

If $\k0=3$, by Proposition \ref{cor n+1} \eqref{c1 n+1} we have $\cc_1^2[\M]=9 N_0$ which, together with \eqref{c1 n=2} and \eqref{NiN},
implies $N_0=N_1=N_2$. The expression for the Hilbert polynomial follows immediately from Proposition \ref{cor n+1}.
The claims in (b) follow similarly by using Proposition \ref{cor n}.

Suppose that $(\M,\omega)$ and the $S^1$-action are symplectic. By Lemma \ref{N0 1} and the fact that $N_0\neq 0$ we have that the action is Hamiltonian and $N_0=1$ (this also reproves
McDuff's theorem \cite{MD1} in the case in which the fixed point set is discrete).
Observe that blowing-up at one fixed point increases the second Betti number $b_2$ by $1$. It follows that the two families of minimal spaces in Remark \ref{minimal spaces}
 are the only compact, connected symplectic manifolds of dimension $4$ that can be endowed with a symplectic circle action with isolated fixed points, with $b_2\leq 2$.
If $\k0=3$, from (a) and \eqref{bi=Ni} we have that $b_0(\M)=b_2(\M)=b_4(\M)=1$, and (a') follows from the classification in \cite{K}.
If $\k0=2$, from (b) we have that $b_0(\M)=b_4(\M)=1$ and $b_2(\M)=2$, and the claim in (b') follows as well from \cite{K}.

\end{proof}
\begin{rmk}\label{not n}
\begin{enumerate}
\item Since symplectic $S^1$-spaces of dimension $4$ are completely classified, the claims in (a') and (b') also follow from the classification in \cite{K}. However we would like to
point out that Proposition \ref{dim 4} (a) and (b) implies immediately that for $\k0=3$ the Betti numbers of $(\M,\omega,S^1)$ are exactly those of $\C P^2$, and for $\k0=2$ they are exactly those of a Hirzebruch surface.
\item The numbers $\lambda,l,m$ appearing in (a') are determined by the `Karshon graph' $\Gamma$ associated to 
$(\M,\omega,S^1)$, as described carefully in \cite{K}; a similar conclusion holds for the case in (b').
\end{enumerate}
\end{rmk}

If $\k0=1$, Proposition \ref{k0=n-1} implies that the Hilbert polynomial
depends on the value of $ \cc_1^2[\M]$. It is interesting to study the position of the roots of $\Hil(z)$ in terms of 
 $\beta=\frac{N_1}{N_0}$. Observe that, by Proposition \ref{dim 4}, $\beta>0$ and if the action is Hamiltonian (and the manifold is connected) then $\beta=b_2(\M)$.
From the definition of Hilbert polynomial of $(\M,\J)$ and \eqref{c1 n=2} (see Proposition \ref{k0=n-1}), it is immediate to see that 
 \begin{equation}\label{Hil n=2 k0=1}
 \Hil(z)=\frac{N_0}{2}\big[(10-\beta)z^2+(10-\beta)z+2\big]\,.
 \end{equation}
 Thus for $\beta\neq 10$ the roots, which are of the form $-\frac{1}{2}\pm a$ with $a$ either real or pure imaginary, have the following position:
 \begin{itemize}
 \item for $0<\beta<2$ or $\beta>10$ they are real and distinct;
 \item for $\beta=2$ they are real and coincide;
 \item for $2<\beta<10$ they live on the axis $-\frac{1}{2}+iy$, for $y\in \R\setminus\{0\}$.
 \end{itemize}
 Moreover when $\left\| \cc_1^2[\M]\right\|\to +\infty$, or equivalently when $\beta\to +\infty$, the roots cluster around the ``foci" $0$ and $-1$.

Observe that by Proposition \ref{dim 4}, in the symplectic case it is impossible to have $\k0=1$ and $\beta=b_2(\M)\leq 2$. 
Moreover, we can have manifolds with $b_2(\M)$ arbitrarily large; it is sufficient to blow-up $\C P^2$ as many times
as we want.

\subsection{$\mathbf{\dim(\M)=6}$} 
Now suppose that $\dim(\M)=6$.  As a consequence of \eqref{NiN} and \eqref{eq cn} we have that
\begin{equation}\label{c3 6}
\cc_3[\M]=2(N_0+N_1)\,,
\end{equation}
and, as a direct consequence of Theorem \ref{nostro}, that
\begin{equation}\label{c1c2 6}
\cc_1\cc_2[\M]=24\,N_0\,.
\end{equation}
\begin{rmk}\label{pos 2}
In dimension $6$ the congruences that must be satisfied by the Chern numbers are
$\cc_1\cc_2[\M]\equiv 0 \mod{24}$, and $\cc_1^3[\M]\equiv \cc_3[\M]\equiv 0 \mod{2}$. Equations \eqref{c3 6} and \eqref{c1c2 6} show that 
for $S^1$-spaces $\cc_1\cc_2[\M]$ is always a \emph{non-negative} multiple of $24$, and $\cc_3[\M]$ a \emph{positive} multiple of $2$. However
our method does not give (in)equalities for $\cc_1^3[\M]$, unless $\k0=3,4$, see Proposition \ref{dim 6}.
\end{rmk}
The following proposition 
follows immediately from
Propositions \ref{cor n+1}, \ref{cor n} and Lemma \ref{N0 1}:
\begin{prop}[$\mathbf{\dim(\M)=6},\;\k0=3,4$]\label{dim 6}
Let $(\M,\J,S^1)$ be an $S^1$-space of dimension $6$, and let
$\k0$, $\Hil(z)$ and the $N_j$'s be defined as before. 
\begin{itemize}
\item[(a)]  If $\k0=4$ then 
$$
\cc_1^3[\M]=64 N_0\quad\;\;\;\mbox{and}\;\;\;\quad \Hil(z)=\frac{N_0}{6}(z+1)(z+2)(z+3).
$$
\item[(b)] If $\k0=3$ then
$$
\cc_1^3[\M]=54 N_0\quad\;\;\;\mbox{and}\;\;\;\quad \Hil(z)=\frac{N_0}{6}(2z+3)(z+1)(z+2).
$$
\end{itemize} 
If we are given $(\M,\omega,S^1)$ of dimension $6$ we have that: 
\begin{itemize}
\item[(i)] If the action is Hamiltonian, then $\k0=4$ implies $(\cc_1^3[\M],\cc_1\cc_2[\M])=(64,24)$, and 
$\k0=3$ implies $(\cc_1^3[\M],\cc_1\cc_2[\M])=(54,24)$.
\item[(ii)] If the action is non-Hamiltonian, then for all $\k0\geq 3$ we have $(\cc_1^3[\M],\cc_1\cc_2[\M])=(0,0)$.
\end{itemize}
\end{prop}

When $\k0<3$, the Chern number $ \cc_1^3[\M]$ and the Hilbert polynomial $\Hil(z)$ are not determined by the index, $N_0$ and $N_1$ (see Remark \ref{not same}).
For example, if $\k0=2$ then from Proposition \ref{k0=n-1} it follows that for $N_0\neq 0$ and $ \cc_1^3[\M]\neq 0$ we have 
\begin{equation}\label{k0=1,2}
\cc_1^3[\M]=\frac{48 N_0}{1-a}\quad\mbox{and}\quad \Hil(z)=\frac{N_0}{1-a}\big[z^2+2z+1-a\big](z+1)
\end{equation}
where $a\neq 1$. 
Thus the roots of $\Hil(z)/(z+1)$ are real exactly if $ \cc_1^3[\M]\geq 48\,N_0\;\;$ or  $\;\;\cc_1^3[\M]<0$.
Moreover they cluster around the ``foci" $0$ and $-2$ exactly if $\left\| \cc_1^3[\M]\right\|\to +\infty$.
\begin{exm}\label{examples 6} In the following we give examples of manifolds of dimension $6$ with $\k0=2$, together with their associated Hilbert polynomials.
\begin{itemize}
\item[(1)] \emph{The flag variety }$\mathcal{F}l(\C^3)=:\mathcal{F}$. The variety of complete flags in $\C^3$ is a compact symplectic (indeed K\"ahler) manifold of dimension $6$ which can be endowed with a Hamiltonian $S^1$-action with exactly $6$ fixed points; for details about the action see \cite[Example 5.5]{GoSa} and the discussion preceding it. The reader can verify that the definition of $\k0$ given here coincides with that of $C$ given in \cite{GoSa}, hence 
$\k0=2$. Moreover $\cc_1^3[\mathcal{F}]=48$, and the Hilbert polynomial is $\Hil_{\mathcal{F}}(z)=(z+1)^3$.
\item[(2)] \emph{The product of spheres $S^2\times S^2\times S^2=:\mathcal{S}$}. This is a compact symplectic (indeed K\"ahler) manifold which can be endowed with a Hamiltonian $S^1$-action with exactly $2^3=8$ fixed points. Moreover it can be checked that
$\cc_1^3[\mathcal{S}]=48$, and the Hilbert polynomial is $\Hil_{\mathcal{S}}(z)=(z+1)^3$.
\item[(3)] \emph{The Fano threefold $V_5$} (for details see \cite{M, T1} or \cite[Example 6.14]{GoSa}). This is a Fano manifold which can be endowed with a Hamiltonian $S^1$-action with exactly $4$ fixed points. The cohomology ring is given by $\Z[x,y]/\langle x^2-5y,y^2 \rangle$ (where $x$ has degree $2$, and $y$ degree $4$), $\k0=2$ and $\cc_1=2x$. Thus $\cc_1^3[V_5]=40$,
and the Hilbert polynomial is $\Hil_{V_5}(z)=\frac{1}{6}\big[5z^2+10z+6\big](z+1)$.
\item[(4)] \emph{A non-K\"ahler example} $n\mathcal{K}$. In \cite{T1}, Tolman constructs a $6$-dimensional compact symplectic manifold  which supports a Hamiltonian action of a $2$-dimensional torus $T$
with isolated fixed points, but does not admit any $T$-invariant K\"ahler structure. Moreover this action is GKM (see \cite{GKM}), and its index $\k0$, as well as the Chern number $\cc_1^3[n\mathcal{K}]$, can be computed from its GKM graph (see \cite{GT}, in particular Example 5.2 and Figure 1, as well as the discussion on page 27 in \cite{GoSa}). It can be checked that in this case $\cc_1=4 \tau_1 + 2 \tau_2$, where $\tau_i\in H^2(\M;\Z)$ is the image under $r_H$ of
the canonical class $\tau_i^{T}\in H_T^2(\M;\Z)$ introduced in \cite{GT}, for $i=1,2$. Since $H^2(\M;\Z)=\Z\langle \tau_1,\tau_2 \rangle$, we have $\k0=2$. Moreover $\cc_1^3[n\mathcal{K}]=64$ and the Hilbert polynomial is $\Hil_{n\mathcal{K}}(z)=\frac{1}{3}\big[4z^2+8z+3\big](z+1)$.
\end{itemize}
\begin{rmk}\label{not same}
Notice that the flag variety in (1) and the non-K\"ahler example in (4) have the same index, the same Betti numbers (hence the same $N_j$'s), but different value of $\cc_1^3[\M]$ and different Hilbert polynomial. 
\end{rmk}
\end{exm}
If $\k0=1$ then from Proposition \ref{k0=n-2} it follows that for $N_0\neq 0$ and $\cc_1^3[\M]\neq 0$ we have 
\begin{equation}\label{k0=1,2 2}
\cc_1^3[\M]=\frac{48 N_0}{1-4a}\quad\mbox{and}\quad \Hil(z)=\frac{N_0}{1-4a}\big[4z^2+4z+1-4a\big](2z+1)
\end{equation}
where $a\neq \frac{1}{4}$. 
Thus the roots of $\Hil(z)/(2z+1)$ are real exactly if $\cc_1^3[\M]\geq 48\,N_0\;\;$ or  $\;\;\cc_1^3[\M]<0$.
Moreover they cluster around the ``foci" $0$ and $-2$ exactly if $\left\|\cc_1^3[\M]\right\|\to +\infty$.

\begin{exm}\label{examples 6-1}
In the following we give examples of manifolds of dimension $6$ with $\k0=1$, together with their associated Hilbert polynomials.
\begin{itemize}
\item[(1)] $\C P^1\times \C P^2=:\mathcal{C}$. This is a compact symplectic (indeed K\"ahler) manifold which can be endowed with a Hamiltonian $S^1$-action with $6$ fixed points. Moreover $\cc_1^3[\mathcal{C}]=54$, and the Hilbert polynomial is $\Hil_{\mathcal{C}}(z)=\frac{1}{2}\big[9z^2+9z+2\big](2z+1)$.
\item[(2)] \emph{The Fano threefold $V_{22}$} (for details see \cite{M, T1} or \cite[Example 6.14]{GoSa}). Similarly to Example \ref{examples 6} (3), this is a Fano manifold which can be endowed with a Hamiltonian $S^1$-action with exactly $4$ fixed points. The cohomology ring is given by $\Z[x,y]/\langle x^2-22y,y^2 \rangle$ (where $x$ has degree $2$, and $y$ degree $4$), $\k0=1$ and $\cc_1=x$. Thus $\cc_1^3[V_{22}]=22$, and the Hilbert polynomial is $\Hil_{V_{22}}(z)=\frac{1}{6}\big[11z^2+11z+6\big](2z+1)$.
\end{itemize}
\end{exm}

\subsection{$\mathbf{\dim(\M)=8}$}
When $\dim(\M)=8$, from \eqref{NiN} and \eqref{eq cn} we have that
\begin{equation}\label{eq c4 8}
  \cc_4[\M]= 2\,N_0+2\,N_1+N_2\,,
\end{equation}
and from Theorem \ref{nostro}
\begin{equation}\label{c1c3 8}
 \cc_1\cc_3[\M]=44\,N_0+8\,N_1-2\,N_2\,.
\end{equation}
As for the remaining Chern numbers, we can use Propositions \ref{cor n+1} and \ref{cor n} to prove the following
\begin{prop}[$\mathbf{\dim(\M)=8},\;\k0=4,5$]\label{dim 8}
Let $(\M,\J,S^1)$ be an $S^1$-space of dimension $8$, and let
$\k0$, $\Hil(z)$ and the $N_j$'s be defined as before. 
\begin{itemize}
\item[(a)]  If $\k0=5$ then
\begin{equation}\label{k0=5 8}
 \cc_1^4[\M]=625\,N_0\,,\quad  \cc_1^2\cc_2[\M]=250\,N_0\,,\quad  \cc_2^2[\M]=101\,N_0-2\,N_1+N_2\,,
\end{equation}
and $\Hil(z)= \displaystyle\frac{N_0}{24}\prod_{j=1}^4(z+j)$\,.
\item[(b)] If $\k0=4$ then
\begin{equation}\label{k0=4 8}
 \cc_1^4[\M]=512\,N_0\,,\quad  \cc_1^2\cc_2[\M]=224\,N_0\,,\quad  \cc_2^2[\M]=98\,N_0-2\,N_1+N_2\,,
\end{equation}
and $\Hil(z)= \displaystyle\frac{N_0}{12}(z+2)\prod_{j=1}^3(z+j)$.
\end{itemize} 
Moreover, if $(\M,\omega)$ is a connected symplectic manifold and the $S^1$-action is Hamiltonian then
\begin{itemize}
\item[(a')] if $\k0=5$ we have
\begin{equation}\label{k0=5 8 1}
 \cc_1^4[\M]=625\,,\quad  \cc_1^2\cc_2[\M]=250\,,\quad  \cc_2^2[\M]=101-2\,b_2(\M)+b_4(\M);
\end{equation}
\item[(b')] if $\k0=4$ we have
\begin{equation}\label{k0=5 8 2}
 \cc_1^4[\M]=512\,,\quad  \cc_1^2\cc_2[\M]=224\,,\quad  \cc_2^2[\M]=98-2\,b_2(\M)+b_4(\M).
\end{equation}
\end{itemize}
If $(\M,\omega)$ is a connected symplectic manifold and the $S^1$-action is non-Hamiltonian then for all $\k0\geq 4$ 
\begin{equation}\label{non-ham 8}
 \cc_1^4[\M]= \cc_1^2\cc_2[\M]=0\quad\mbox{and}\quad   \cc_2^2[\M]=-2N_1+N_2
\end{equation}
\end{prop}
\begin{proof}
The only claims in \eqref{k0=5 8} and \eqref{k0=4 8} which do not follow directly from Propositions \ref{cor n+1} and \ref{cor n} are the expressions of $ \cc_2^2[\M]$ in terms of the $N_j$'s. In order to obtain them, it is sufficient to use the expression of the Todd genus given in Corollary \ref{todd genus comp},
which for $n=4$ gives
\begin{equation}\label{todd 4}
 \frac{-\cc_1^4+4\cc_1^2\cc_2+3\cc_2^2+\cc_1\cc_3-\cc_4}{720}[\M]=N_0\,.
\end{equation}
By combining \eqref{todd 4} with \eqref{c1 n+1}, \eqref{c1c2}, \eqref{c1 n} and \eqref{c1c22} we obtain the desired claims.
In the symplectic case, all the claims follow from Lemma \ref{N0 1}, \eqref{bi=Ni}, Corollary \ref{bound on k0} and \eqref{todd 4}.

\end{proof}
When $\k0=3$ or $\k0=2$, from Proposition \ref{k0=n-1} and \ref{k0=n-2} we can see that the coefficients of the Hilbert polynomial depend on the value of $ \cc_1^4[\M]$. The following proposition exhibits the
relation between $ \cc_1^2\cc_2[\M]$, $ \cc_2^2[\M]$ and $ \cc_1^4[\M]$.
\begin{prop}[$\mathbf{\dim(\M)=8},\;\k0=2,3$]\label{dim 8 2}
Let $(\M,\J,S^1)$ be an $S^1$-space of dimension $8$, and let
$\k0$, $\Hil(z)$ and the $N_j$'s be defined as before. Then 
\begin{itemize}
\item[(a)] $\k0=3$ implies that 
\begin{equation}
\label{k0=3 c1c2}  \cc_1^2\cc_2[\M]=108\,N_0+\frac{2}{9} \cc_1^4[\M]\,,
\end{equation}
and
\begin{equation}
\label{k0=3 c22}  \cc_2^2[\M]=82\,N_0-2\,N_1+N_2+\frac{1}{27} \cc_1^4[\M]\,.
\end{equation} 

\item[(b)] $\k0=2$ implies that 
\begin{equation}
\label{k0=2 c1c2}  \cc_1^2\cc_2[\M]=96\,N_0+\frac{1}{4} \cc_1^4[\M]\,,
\end{equation}
and
\begin{equation}
\label{k0=2 c22}  \cc_2^2[\M]=98\,N_0-2\,N_1+N_2\,.
\end{equation} 

\end{itemize}

\end{prop}

\begin{proof}

(a) In order to prove \eqref{k0=3 c1c2}, it is sufficient to use Corollary \ref{relation c122}, and
equation \eqref{k0=3 c22} can be obtained by combining \eqref{todd 4} with \eqref{eq c4 8}, \eqref{c1c3 8} and \eqref{k0=3 c1c2}.

(b) Equation \eqref{k0=2 c1c2} follows from Corollary \ref{relation c122 2}, and
\eqref{k0=2 c22} can be obtained by combining \eqref{todd 4} with \eqref{eq c4 8}, \eqref{c1c3 8} and \eqref{k0=2 c1c2}.
\end{proof}
We conclude this section with the following corollary:
\begin{corollary}\label{c228h}
Let $(\M,\omega)$ be a compact, connected symplectic manifold of dimension $8$ that can be endowed with a Hamiltonian circle action with
isolated fixed points. If the minimal Chern number is \emph{even}, then 
$$
\cc_2^2[\M]+2\,b_2(\M)=98+b_4(\M)\,.
$$
\end{corollary}
\begin{proof}
If $(\M,\omega)$ can be endowed with a Hamiltonian circle action with isolated fixed points, then by Corollary \ref{minimal chern ham} the minimal Chern number coincides with the index, and it can be only $1,2,3,4$ or $5$. Since it is even, the claim follows from  
\eqref{k0=5 8 2}, \eqref{k0=2 c22} and \eqref{bi=Ni}.
\end{proof}

\begin{rmk}\label{pos 8}
It is easy to check that all the necessary congruences among the Chern numbers for $n=4$ are satisfied; in particular 
$(-\cc_1^4+4\cc_1^2\cc_2+\cc_1\cc_3+3\cc_2^2-\cc_4)[\M]$ must be a non-negative multiple of $N_0$. 
If $\k0\geq 4$ then $(2\cc_1^4+\cc_1^2\cc_2)[\M]$ must be a non-negative multiple of $12$. However, in general,
we cannot conclude such non-negativity results.
\end{rmk}

\begin{rmk}\label{comparison}
It can be checked that \eqref{k0=2 c1c2} is equivalent to equation (7.22) in \cite{GoSa}; here $\M$ is an $8$-dimensional
compact symplectic manifold,
with a Hamiltonian $S^1$-action and exactly $5$ fixed points.
Equation (7.22) in \cite{GoSa} is obtained by applying some results of Hattori (see \cite{Ha}, and Corollary 7.7, Theorem 7.11 in \cite{GoSa}) which, however, only hold 
whenever $(\M,\J)$ possesses a fine line bundle. Moreover, the derivation of (7.22) from such results is rather complicated, as it can be seen
from the proof of \cite[Theorem 7.11]{GoSa}.
Here we do not need to assume the existence of
a fine line bundle, and \eqref{k0=2 c1c2} is an immediate consequence of Corollary \ref{relation c122 2}.
\end{rmk}

When $\k0=1$ we do not obtain any restrictions on the Chern numbers (see Corollary \ref{cor equations chern numbers}, as well as
the discussion on the case $\k0=n-3$ at the end of Section \ref{sec: values k0}).

\end{document}